\documentclass[a4paper,reqno]{amsart}

\usepackage[utf8]{inputenc}
\usepackage[oneside]{typearea}
\usepackage{mathtools}
\usepackage{pdfpages}

\newcommand{\inn}{\operatorname{in}}

\newcommand{\transpose}{^T}
\newcommand{\Eins}{\mathbf 1}
\newcommand{\ess}{\operatorname{ess}}
\newcommand{\sgn}{\operatorname{sgn}}
\newcommand{\Id}{\mathbf{Id}}

\newcommand{\init}{\operatorname{init}}
\newcommand{\ter}{\operatorname{ter}}
\newcommand{\out}{\operatorname{out}}
\renewcommand{\div}{\operatorname{div}}

\newtheoremstyle{theorem_fett}
 {\topsep}                    
    {\topsep}                    
    {\itshape}                   
    {}                           
    {}                   
    {.}                          
    {.5em}                       
  {{\bfseries \thmname{#1}\thmnumber{ #2}\thmnote{ (#3)}}}
  \theoremstyle{theorem_fett}
\newtheorem{theorem}{Theorem}
\newtheorem{lemma}[theorem]{Lemma}
\theoremstyle{remark}
\newtheorem{remark}[theorem]{Remark}
\theoremstyle{definition}
\newtheorem{Def}[theorem]{Definition}

\usepackage{hyperref}
\title[Transport equations on a network]{Well-posedness of space and time dependent transport equations on a network}
\author{Arne Roggensack}
\address{Weierstrass Institute, Mohrenstr. 39, 10117 Berlin, Germany}
\email{arne.roggensack@wias-berlin.de}
\date{\today}

\begin{document}
\begin{abstract}
This article is concerned with the study of weak solutions of a linear transport equation on a bounded domain with coupled boundary data for general non smooth space and time dependent velocity fields. The existence of solutions, its uniqueness and the continuous dependence of the solution on the initial and boundary data as well on the velocity is proven. The results are based on the renormalization property. At the end, the theory is shown to be applicable to the continuity equation on a network.
\end{abstract}

\maketitle

\section{Introduction}
\label{chap:transport}
In this article, we study linear transport equations on bounded domains with coupled boundary values and prove their well-posedness. As a main application we have in mind transport phenomena on a network which occur in various fluid dynamical applications. In that case, on each edge of a graph a linear transport equation is imposed. The equations are coupled at the nodes by linear boundary conditions mapping from the outflow part of the boundary to the inflow part. From an abstract point of view, one can see this as a multivariate transport equation on a one-dimensional domain with coupled boundary data
\begin{align*}
\notag
	\rho_t+(\mathbf U\rho)_x+\mathbf C\rho&=f&\text{in }(0,T)\times\Omega\\
	\rho(0,\cdot)&=\rho_0&\text{in }\Omega\\
	(\nu \mathbf U)^-\rho&=(\nu \mathbf U)^-\mathcal H(\rho|_{\Gamma_T})\quad&\text{on }\Gamma_T.
\end{align*} 

We will prove existence, boundedness, uniqueness, and stability of solutions $\rho\in C([0,T],L^p((0,1))^n)$ for non-smooth matrix-valued velocities $\mathbf U=\operatorname{diag}(u)$ with $u\in L^1((0,T),W^{1,1}((0,1))^n)$ and $u_x\in L^1((0,T),L^{\infty}((0,1))^n)$ without any restrictions on the sign of the velocity. In the Lipschitz continuous setting, every change of the sign of the velocity implies a change of the direction of a characteristic. This can lead to additional difficulties since it has two effects:

 First, even if the data is smooth, there does not have to exist a classical smooth solution as the domain is bounded. For example, by the method of characteristics the solution of
\begin{equation*}
\rho_t+\left((2t-1)\rho\right)_x=0
\end{equation*}
on $(0,1)\times(0,1)$ is given by
\begin{equation*}
\rho(t,x)=\begin{cases}
\rho_0\left(x-t(t-1)\right)&\text{if }t<\frac 12\text{ or }x\geq \left(t-\frac12\right)^2\\
\rho_{\inn}\left(\frac 12+\sqrt{\left(t-\frac12\right)^2-x}\right)&\text{otherwise,}
\end{cases}
\end{equation*}
which has possibly a discontinuity if no further conditions are imposed. The only way to get smooth solutions is to require additional conditions on the initial and boundary data $\rho_0$ and $\rho_{\inn}$, which depend on the characteristics and thus on the velocity. In many applications however, the velocity is not known a priori.

The second issue concerns the coupled boundary data. In this case, we can compute the characteristics, but due to the coupling conditions, we are not able to use them directly to find a solution if the velocity has an infinite number of changes of the sign in a finite time. On a network, this can lead to a characteristic of the form of an infinite tree. Each time the characteristic intersects an inner node it divides itself into multiple parts.

For these reasons we will use a different and very general approach looking for weak solutions. This approach is based on the concept of the renormalization property, which was introduced in 1989 by Lions and DiPerna \cite{DiPerna.1989} for tangential velocity fields. This concept was extended to bounded domains in $\mathbb R^N$ with inflow boundary conditions and velocity fields with a kind of Sobolev regularity by Boyer \cite{boyer2005} in 2005 and Boyer and Fabrie \cite{Boyer.2012}. Recently, these results were generalized to velocity fields with $BV$ regularity by Crippa et al.\ \cite{Crippa.2013a,Crippa.2013b} and to stochastic differential equations by Neves and Olivera \cite{Neves.2014}. For constant velocity fields, the transport equation on a network was studied by Sikolya in her dissertation \cite{Sikolya2004} in 2004. In 2008, this was expanded to infinite networks by Dorn \cite{Dorn.2008, Dorn.2008a}. Especially, the long time behaviour was analysed by a semigroup approach.

In this article, we generalize the results of Boyer and Fabrie to the case of coupled boundary data. In particular, we will study a vector valued transport equation with affine linear coupled boundary conditions and a time and space dependent velocity field.

Our results for the transport equation can e.g.\ be applied to study the existence of solutions of a low Mach number model on a network. Details to this topic can be found in a forthcoming paper or in the PhD thesis \cite{RoggensackDiss.2014}. However, transport equations on networks occur also in various other contexts as traffic flow \cite{Holden.1995,garavello2006}, gas flow in pipe lines \cite{Banda.2006} or river flow \cite{cunge1980,roggensack2013}
 
 The article is divided in six sections. In Section \ref{sec:transport_assump}, we introduce the notations and requirements. Section \ref{sec:uniqueness} recalls the renormalization property and it provides uniqueness and boundedness results. In section \ref{sec:cont_depend}, we prove the continuous dependence of the solution on the data before we show the existence of a solution in Section \ref{sec:existence}. In the last section, we illustrate how these general results can be applied to the continuity equation on a network.
\section{Assumptions and requirements}
\label{sec:transport_assump}
In this section, we introduce the precise setting. For simplicity, we will only consider the one-dimensional domain $\Omega=(0,1)$ and denote the boundary $\Gamma=\partial \Omega=\{0,1\}$ and $\Gamma_T=[0,T]\times\Gamma$. Since we have in mind to consider one-dimensional edges of a graph this is no restriction for the application. The outer normal vector on $\Gamma$ is called $\nu$ and is given by
$\nu=-(-1)^{\omega}$
for $\omega\in\Gamma$. 	We use the integral notation $\int_{\Gamma}$ although the boundary consists only of two single points, i.e.\ $\mathrm d\omega$ is the counting measure. An advantage of this notation is the easy generalization to higher dimensions. All considerations in this article could also be done for the case of bounded multidimensional domains $\Omega$ with Lipschitz boundaries.

The aim of this article is to find and characterize a solution $\rho(t,x)\in \mathbb R^n$ of the problem
\begin{align}
\notag
	\rho_t+(\mathbf U\rho)_x+\mathbf C\rho&=f&\text{in }(0,T)\times\Omega\\
	\label{eq:transport_equation}
	\rho(0,\cdot)&=\rho_0&\text{in }\Omega\\
	\notag
	(\nu \mathbf U)^-\rho&=(\nu \mathbf U)^-\mathcal H(\rho|_{\Gamma_T})\quad&\text{on }\Gamma_T
\end{align} 
with initial conditions $\rho_0\in L^{\infty}(\Omega)^n$ with $\rho_0\geq 0$ almost everywhere and with an affine linear boundary operator $\mathcal H$. We assume
\begin{align}
	u&\in L^1((0,T),W^{1,1}(\Omega)^{n}),\notag\\
	c&\in L^1((0,T)\times\Omega)^n,\notag\\
	(u_x+c)^-&\in L^1((0,T),L^{\infty}(\Omega)^n),\label{assump:ucf}\\
	(u_x)^+&\in L^1((0,T),L^{\infty}(\Omega)^n)\notag\\
		\shortintertext{and}
	f&\in L^1((0,T),L^{\infty}(\Omega)^n)\notag
\end{align}
with $f\geq 0$ and denote by $\mathbf U=\operatorname{diag}(u)$ and $\mathbf C=\operatorname{diag}(c)$ diagonal matrices with the diagonal entries $u^j$ and $c^j$, respectively. On $\Gamma_T$ we define component-by-component the measure
\begin{align}
\mathrm d\mu_{u}=(\nu u)\mathrm d\omega\mathrm dt
\end{align}
 and introduce its positive part $\mathrm d\mu_{u}^+=(\nu u)^+\mathrm d\omega\mathrm dt$, its negative part $\mathrm d\mu_{u}^-=(\nu u)^-\mathrm d\omega\mathrm dt$ and its absolute value $|\mathrm d\mu_{u}|=\mathrm d\mu_u^++\mathrm d\mu_u^-$. For each $j$, these measures divide the boundary $\Gamma_T$ into two parts, the inflow part $\Gamma_T^-$ with $\nu u_j<0$ and the outflow part $\Gamma_T^+$ with $\nu u_j>0$. For $p\in[1,\infty)$ the space $L^p(\Gamma_T,\mathrm d\mu_u^{\pm})=L^p(\Gamma_T,\mathbb R^n,\mathrm d\mu_u^{\pm})$ is provided with the norm
\begin{equation*}
	\|g\|_{p,w,\pm}=\left(\int_0^T\int_{\Gamma}\beta(|g|)\transpose\overline{\mathbf W}(\nu u)^{\pm}\mathrm d\omega\mathrm dt\right)^{\frac 1p}
\end{equation*}
with the function $\beta(x)=\begin{pmatrix}
x_1^p&\cdots&
x_n^p\end{pmatrix}\transpose$ and with a positive definite diagonal weight matrix $\overline{\mathbf W}\in \mathbb R^{n\times n}$. Here, the absolute value has to be understood component-by-component. We do also consider the space $L^{\infty}(\Gamma_T,\mathrm d\mu_u^{\pm})=L^{\infty}(\Gamma_T,\mathbb R^n,\mathrm d\mu_u^{\pm})$ equipped with the maximum norm of the componentwise $L^{\infty}$-norm, i.e.\ 
\begin{equation*}
	\|g\|_{\infty,\pm}=\max_j\left(\ess\sup_{\substack{{(\omega,t)\in\Gamma_T}\\{u^j(t,\omega)^{\pm}\neq 0}}}|g^j|\right).
\end{equation*}

The boundary operator $\mathcal H$ has to assign a boundary value to the inflow part of the boundary, i.e.\ $\mathcal H$ is a mapping
\begin{equation*}
\mathcal H\colon L^{\infty}(\Gamma_T,\mathrm d\mu_{u}^+)\to L^{\infty}(\Gamma_T,\mathrm d\mu_u^-).
\end{equation*}
We assume the mapping to be affine linear, i.e.\ it is given as $\mathcal H(\rho)=\rho_{\inn}+\mathcal G(\rho)$ where it is $\rho_{\inn}\in L^{\infty}(\Gamma_T,\mathrm d\mu_u^-)$ with $\rho_{\inn}(t,x)\geq 0$ for $\mathrm d\mu_u^-$-almost all $(t,\omega)$. Here, $\mathcal G$ is a linear operator fulfilling the following conditions:
\begin{itemize}
	\item[-] The operator $\mathcal G\colon L^{\infty}(\Gamma_T,\mathrm d\mu_u^+)\to L^{\infty}(\Gamma_T,\mathrm d\mu_u^-)$ is weakly-$\star$ continuous, i.e.\ for all $\rho_n\in L^{\infty}(\Gamma_T,\mathrm d\mu_u^+)$ with $\rho_n\xrightharpoonup{\star}\rho$ in $L^{\infty}(\Gamma_T,\mathrm d\mu_u^+)$ it holds
	\begin{equation*}
	\mathcal G(\rho_n)\xrightharpoonup{\star}\mathcal G(\rho)
	\end{equation*}
	in $L^{\infty}(\Gamma_T,\mathrm d\mu_u^-)$.
	\item[-] There is a positive definite diagonal weight matrix $\overline{\mathbf W}\in \mathbb R^{n\times n}$ such that the $L^1$-operator norm of $\mathcal G$ is less or equal one, i.e.\  for all $\rho\in L^{\infty}(\Gamma_T,\mathrm d\mu_u^+)$ it holds
	\begin{equation*}
		\|\mathcal G(\rho)\|_{1,w,-}\leq \|\rho\|_{1,w,+}.
	\end{equation*}
	\item[-] >The operator $\mathcal G$ is causal, i.e.\ for all $\rho\in L^{\infty}(\Gamma_T,\mathrm d\mu_u^+)$ and almost all $t\in[0,T]$ it holds
	\begin{equation}
	\label{eq:G_chi}
		\chi_{[0,t]}\mathcal G(\chi_{[0,t]}\rho)=\chi_{[0,t]}\mathcal G(\rho).
	\end{equation}
	\item[-]
	The operator $\mathcal G$ is $\mathrm d\mu_u^+$-almost everywhere positive, i.e.\  it holds $\mathcal G(\rho)\geq 0$ $\mathrm d\mu_u^-$-almost everywhere  for all $\rho\in L^{\infty}(\Gamma_T,\mathrm d\mu_u^+)$ with $\rho\geq 0$ $\mathrm d\mu_u^+$-almost everywhere.
	\item[-]
	There exists a constant vector $\rho_{\max}\in \mathbb R^n$ such that component-by-component the inequalities
	\begin{equation}
	\label{eq:exist_M}
		\rho_{\inn}\exp\left(-\int_0^t\alpha(s)\mathrm ds\right)-F_++\mathcal G(F_+)+\mathcal G(\rho_{\max})\leq \rho_{\max}
	\end{equation}
	for $\mathrm d\mu_u^-$-almost all $(t,\omega)\in \Gamma_T$ and
	\begin{equation*}
		\rho_0\leq \rho_{\max}
	\end{equation*}
	for almost all $x\in \Omega$ are valid. Here, it is $\alpha(s)=\|(u_x(s,\cdot)+c(s,\cdot))^-\|_{L^{\infty}(\Omega)^n}$ and 
	\begin{equation*}
	F_+(t)=\int_0^t\exp\left(-\int_0^s\alpha(r)\mathrm dr\right)\begin{pmatrix}
	\|(f_1(s,\cdot))^+\|_{L^{\infty}}\\
	\vdots\\
		\|(f_n(s,\cdot))^+\|_{L^{\infty}}
	\end{pmatrix}\mathrm ds.
	\end{equation*}
\end{itemize}
As we will see later, this construction of $\mathcal H$ is motivated by the transport equation on a network.
 
To conclude this section, we observe three properties of the operator $\mathcal G$.
First, for all $\rho\in L^{\infty}(\Gamma_T,\mathrm d\mu_u^+)$ and all $t\in[0,T]$ the inequality
\begin{equation}
\begin{split}
\label{eq:observation}
	\int_0^t\int_{\Gamma}|\mathcal G(\rho)|\transpose\overline{\mathbf W}(\nu u)^-\mathrm d\omega\mathrm dt&=\int_0^T\int_{\Gamma}\chi_{[0,t]}|\mathcal G(\rho)|\transpose\overline{\mathbf W}(\nu u)^-\mathrm d\omega\mathrm dt\\
	&=\int_0^T\int_{\Gamma}|\mathcal G(\chi_{[0,t]}\rho)|\transpose\overline{\mathbf W}(\nu u)^-\mathrm d\omega\mathrm dt\\
	&=\|\mathcal G(\chi_{[0,t]}\rho)\|_{1,w,-}\\
	&\leq\|\chi_{[0,t]}\rho\|_{1,w,+}\\
	&= \int_0^T\int_{\Gamma}|\chi_{[0,t]}\rho|\transpose\overline{\mathbf W}(\nu u)^+\mathrm d\omega\mathrm dt\\
	&=\int_0^t\int_{\Gamma}|\rho|\transpose\overline{\mathbf W}(\nu u)^+\mathrm d\omega\mathrm dt
\end{split}
\end{equation}
is true.
Second, the equation \eqref{eq:G_chi} holds in fact for all essential bounded functions $g\in L^{\infty}([0,T])$: Let $\rho\in L^{\infty}(\Gamma_T,\mathrm d\mu_u^+)$, $\rho\neq 0$ and $\varepsilon>0$ be given. Since the step functions are dense in $L^{\infty}([0,T])$ we can choose a step function $g_n=\sum_{k=1}^n a_k\chi_{[t_k,t_{k+1}]}$ with
\begin{equation*}
	\|g-g_n\|_{L^{\infty}([0,T])}<\frac{\varepsilon}{2\|\rho\|_{1,w,+}}.
\end{equation*}
Then, because of $\mathcal G(f_n\rho)=f_n\mathcal G(\rho)$ it follows
\begin{align*}
\|\mathcal G(g\rho)-g\mathcal G(\rho)\|_{1,w,-}&\leq \|\mathcal G((g-g_n)\rho)\|_{1,w,-}+\|(g-g_n)\mathcal G(\rho)\|_{1,w,-}\\
&\leq\|(g-g_n)\rho\|_{1,w,+}+\|g-g_n\|_{L^{\infty}([0,T])}\|\mathcal G(\rho)\|_{1,w,-}\\
&\leq \|g-g_n\|_{L^{\infty}([0,T])}\|\rho\|_{1,w,+}+\|g-g_n\|_{L^{\infty}([0,T])}\|\rho\|_{1,w,+}\\
&<\varepsilon
\end{align*}
and thus 
\begin{equation}
\label{eq:G_continuous}
	\mathcal G(g\rho)=g\mathcal G(\rho)
\end{equation}
holds $\mathrm d\mu_u^-$-almost everywhere.

The third observation concerns the positivity of the operator. Using the positivity of $\mathcal G$ and the triangle inequality, we estimate for $\rho\in L^{\infty}(\Gamma_T,\mathrm d\mu_u^+)$
\begin{equation}
\label{obs:third}
\begin{split}
\left(\mathcal G(\rho)\right)^+&=\frac 12\left(\mathcal G(\rho)+\left|\mathcal G(\rho)\right|\right)\\
&=\frac 12\left(\mathcal G(\rho^+)-\mathcal G(\rho^-)+\left|\mathcal G(\rho^+)-\mathcal G(\rho^-)\right|\right)\\
&\leq\frac 12\left(\mathcal G(\rho^+)-\mathcal G(\rho^-)+\left|\mathcal G(\rho^+)\right|+\left|\mathcal G(\rho^-)\right|\right)\\
&=\mathcal G(\rho^+)
\end{split}
\end{equation}
almost everywhere in $\Gamma_T$.
In the same way, we can prove that
\begin{equation}
\label{osb:third_a}
\left(\mathcal G(\rho)\right)^-\leq\mathcal G(\rho^-)
\end{equation}
holds.
In the following, we often extend functions in $L^p(\Gamma_T,\mathrm d\mu_u^{\pm})$ by zero to construct functions in $L^p(\Gamma_T,|\mathrm d\mu_u|)$ or $L^p(\Gamma_T)^n$. Especially, we extend, if necessary, the operators $\mathcal G$ and $\mathcal H$ to operators mapping from $L^{\infty}(\Gamma_T,|\mathrm d\mu_u|)\to L^{\infty}(\Gamma_T,|\mathrm d\mu_u|)$.

\section{Uniqueness and Boundedness}
\label{sec:uniqueness}
In this section, we will prove the uniqueness and boundedness of solutions of the initial-boundary-value problem \eqref{eq:transport_equation}. To this end, we first introduce the concept of weak solutions before we discuss regularity and the renormalization property of the solutions. Equipped with this, we are able to prove uniqueness and boundedness. For the latter, we will give explicit lower and upper bounds. 

We start with the usual definition for a weak solution.
\begin{Def}
	We call a function
	\begin{equation*}
		\rho\in L^{\infty}([0,T]\times\Omega)^n
	\end{equation*}
	a \emph{weak solution of the transport equation} $\rho_t+(\mathbf U\rho)_x+\mathbf C\rho=f$ if it holds
	\begin{equation*}
	0=\int_0^T\int_{\Omega}\rho\transpose\left(\varphi_t+\mathbf U\varphi_x -\mathbf C\varphi\right)+f\transpose\varphi\mathrm dx\mathrm dt 
	\end{equation*}
	for all $\varphi\in C^{0,1}([0,T]\times\overline{\Omega},\mathbb R^n)$ with $\varphi(0,\cdot)=\varphi(T,\cdot)=0$ and $\varphi=0$ on $\Gamma_T$.
\end{Def}
A priori, it is not clear how to define boundary conditions for this class of weak solutions since $\rho$ is not a Sobolev function and thus, the usual trace theory is not applicable.

In \cite{Boyer.2012}, Boyer has proven some properties of such weak solutions for scalar equations, especially the existence of a trace and the renormalization property of the solution (see the following theorem).
\begin{theorem}[Trace theorem\normalfont{, Boyer \cite{Boyer.2012}}]
	\label{theo:trace}
	Let $\Omega\subset \mathbb R^d$ be a $d$-dimensional bounded Lipschitz domain and let be $u\in L^1((0,T),W^{1,1}(\Omega)^d)$ and $c,f\in L^1((0,T)\times\Omega)$ with $(c+u_x)^-\in L^1((0,T),L^{\infty}(\Omega))$ and $(u_x)^+\in L^1((0,T),L^{\infty}(\Omega))$. Then, for each weak solution $\rho\in L^{\infty}((0,T)\times\Omega)$ of the scalar transport equation
	\begin{equation*}
	\rho_t+\div(u\rho)+c\rho=f
	\end{equation*}
	the following properties hold:
	\begin{description}
		\item[\textbf{Time continuity}] The function $\rho$ lies in $C([0,T],L^p(\Omega))$ for all $p\in[1,\infty)$.
		\item[\textbf{Existence and uniqueness of a trace function}] There exists a unique essentially bounded function $\gamma\rho\in  L^{\infty}(\Gamma_T,|\mathrm d\mu_u|)$, called \emph{trace}, such that for any $[t_0,t_1]\subset[0,T]$ and for all test functions $\varphi\in C^{0,1}([0,T]\times\overline{\Omega})$
		\begin{equation}
		\label{eq:trace_exist}
\begin{split}
			0&=\int_{t_0}^{t_1}\int_{\Omega}\rho(\varphi_t+u\transpose\nabla\varphi-c\varphi)+f\varphi\mathrm dx\mathrm dt-\int_{t_0}^{t_1}\int_{\Gamma}\gamma\rho\varphi(u\transpose\nu)\mathrm d\omega\mathrm dt\\
			&+\int_{\Omega}\rho(t_0)\varphi(t_0)\mathrm dx-\int_{\Omega}\rho(t_1)\varphi(t_1)\mathrm dx
\end{split}
		\end{equation}
		holds.
		\item[\textbf{Renormalization property}]
		For any continuous  and piecewise $C^1$ function $\beta$ the renormalization property holds, i.e.\ $\rho$ satisfies for any $[t_0,t_1]\subset[0,T]$ and for all test functions $\varphi\in C^{0,1}([0,T]\times\overline{\Omega})$
		\begin{equation}
		\label{eq:renorm_prop}
			\begin{split}
			0&=\int_{t_0}^{t_1}\int_{\Omega}\beta(\rho)(\varphi_t+u\transpose\nabla\varphi)\mathrm dx\mathrm dt-\int_{t_0}^{t_1}\int_{\Omega}\beta'(\rho)\left(\rho c-f\right)\varphi\mathrm dx\mathrm dt\\
			&-\int_{t_0}^{t_1}\int_{\Omega}\varphi \operatorname{div}(u)(\beta'(\rho)\rho-\beta(\rho))\mathrm dx\mathrm dt -\int_{t_0}^{t_1}\int_{\Gamma}\beta(\gamma\rho)(u\transpose\nu)\varphi\mathrm d\omega\mathrm dt\\
			&+\int_{\Omega}\beta(\rho(t_0))\varphi(t_0)\mathrm dx-\int_{\Omega}\beta(\rho(t_1))\varphi(t_1)\mathrm dx.
			\end{split}
		\end{equation}
	\end{description}
	\end{theorem}
	\begin{proof} See \cite{Boyer.2012}.
	\end{proof}
	\begin{remark}
		A component-by-component consideration shows the validity of Theorem~\ref{theo:trace} for vector-valued equations with $n>1$. The renormalization property then reads:
		
		There exists a trace $\gamma\rho\in L^{\infty}(\Gamma_T,|\mathrm d\mu_u|)$ such that for any continuous and piecewise $C^1$ function $\beta\colon\mathbb R^n\to\mathbb R^n$ with $D\beta$ diagonal, for any $[t_0,t_1]\subset[0,T]$ and for any test function $\varphi\in C^{0,1}([0,T]\times\overline{\Omega},\mathbb R^n)$ it holds
		\begin{align*}
			0&=\int_{t_0}^{t_1}\int_{\Omega}\beta(\rho)\transpose(\varphi_t+\mathbf U\varphi_x)\mathrm dx\mathrm dt-\int_{t_0}^{t_1}\varphi\transpose\left(\mathbf C D\beta(\rho)\rho-D\beta(\rho)f\right)\mathrm dx\mathrm dt\\
			&-\int_{t_0}^{t_1}\int_{\Omega}\varphi\transpose\mathbf U_x(D\beta(\rho)\rho-\beta(\rho))\mathrm dx\mathrm dt -\int_{t_0}^{t_1}\int_{\Gamma}\beta(\gamma\rho)\transpose(\nu\mathbf U)\mathrm \varphi d\omega\mathrm dt\\
			&+\int_{\Omega}\beta(\rho(t_0))\transpose\varphi(t_0)\mathrm dx-\int_{\Omega}\beta(\rho(t_1))\transpose\varphi(t_1)\mathrm dx.
		\end{align*}
	\end{remark}
	This theorem is no existence result, it only classifies solutions. 
	But due to the continuity of $\rho$ with values in $L^p(\Omega)$ and the existence of the trace we can define a solution of the initial-boundary-value problem as following:
	\begin{Def}
		The function
		\begin{equation*}
			\rho\in L^{\infty}((0,T)\times\Omega)^n
		\end{equation*}
		is called a solution of the initial-boundary-value problem \eqref{eq:transport_equation} if and only if
		\begin{itemize}
			\item[-] $\rho$ is a weak solution of the transport equation,
			\item[-] the initial conditions are fulfilled, i.e.\ $\rho(0)=\rho_0$ and
			\item[-] the trace $\gamma\rho$ satisfies the boundary conditions, i.e.\ $\mathcal H(\gamma\rho)=\gamma\rho.$
		\end{itemize}

	\end{Def}
As a consequence of the renormalization property we can prove the uniqueness of the above defined solution.
	\begin{theorem}[Uniqueness]
		Under the assumptions from Section~\ref{sec:transport_assump}, there is at most one solution of the initial-boundary-value problem \eqref{eq:transport_equation}.
	\end{theorem}
	\begin{proof}
		Let  $\rho^1$ and $\rho^2$ be two solutions of \eqref{eq:transport_equation} with its traces $\gamma\rho_1$ and $\gamma\rho_2$ and define $\rho=\rho^1-\rho^2$. Then, $\rho$ is a weak solution of the homogeneous transport equation.
		Because of the uniqueness of the trace, $\gamma\rho=\gamma\rho_1-\gamma\rho_2$ is the trace of $\rho$. Thus, we see that $\rho$ is a solution of the following homogeneous initial-boundary-value problem
		\begin{align*}
			\rho_t+(\mathbf U\rho)_x+\mathbf C\rho&=0&&\text{in }(0,T)\times\Omega\\
			\rho(0,\cdot)&=0&&\text{in }\Omega\\
			(\nu \mathbf U)^-\rho&=(\nu \mathbf U)^-\mathcal G(\rho|_{\Gamma_T})\ &&\text{on }\Gamma_T.
		\end{align*}
		Using the renormalization property for $\beta(s)=|s|$ and $\varphi=\overline{\mathbf W}\Eins$ with $\Eins=\begin{pmatrix}
				1&\cdots&1
				\end{pmatrix}\transpose$ we conclude for all $t\in [0,T]$
		\begin{align*}
		0&=-\int_{\Omega}\Eins\transpose\overline{\mathbf W}\beta(\rho(t))\mathrm dx-\int_0^{t}\int_{\Omega}\Eins\transpose\overline{\mathbf W}\mathbf C\beta(\rho)\mathrm dx\mathrm dt\\
		&\quad-\int_0^{t}\int_{\Gamma}\Eins\transpose\overline{\mathbf W}(\nu\mathbf U)\beta(\gamma\rho)\mathrm d\omega \mathrm dt\\
		&=-\|\overline{\mathbf W}\rho(t)\|_{L^{1}(\Omega)^n}-\int_0^t\int_{\Omega}\Eins\transpose\overline{\mathbf W}\mathbf C|\rho|\mathrm dx\mathrm dt\\
		&\quad-\int_0^t\int_{\Gamma}|\gamma\rho|\transpose\overline{\mathbf W}(\nu u)^+- |\mathcal G(\gamma\rho)|\transpose\overline{\mathbf W}(\nu u)^-\mathrm d\omega\mathrm dt.
		\end{align*}
		With inequality \eqref{eq:observation} we get
		\begin{equation*}
		\|\overline{\mathbf W}\rho(t)\|_{L^{1}(\Omega)^n}\leq -\int_0^{t}\int_{\Omega}\Eins\transpose\overline{\mathbf W}\mathbf C|\rho|\mathrm dx\mathrm dt
		\end{equation*}
		and thus
		\begin{align*}
			\|\overline{\mathbf W}\rho(t)\|_{L^{1}(\Omega)^n}&\leq\int_0^{t}\|c^-\|_{L^{\infty}(\Omega)^n}\|\overline{\mathbf W}\rho\|_{L^{1}(\Omega)^n}\mathrm dt\\
			&\leq\int_0^{t}\Bigl(\|(u_x+c)^-\|_{L^{\infty}(\Omega)^n}+\|(u_x)^+\|_{L^{\infty}(\Omega)^n}\Bigr)\|\overline{\mathbf W}\rho\|_{L^{1}(\Omega)^n}\mathrm dt.
		\end{align*}
		By Gronwall's inequality we conclude $\|\overline{\mathbf W}\rho(t)\|_{L^{1}(\Omega)^n}=0$ for  all $t$ and thus $\rho=0$ almost everywhere since $\overline{\mathbf W}$ is positive definite. 
	\end{proof}
	\begin{remark}
		\label{remark:f_L1}
		We can slightly weaken the assumptions on the source term $f$. In fact, for the existence of the trace, the renormalization property and the uniqueness it is sufficient to require $f\in L^1((0,T)\times\Omega)^n$ (see \cite{Boyer.2012}).
	\end{remark}
	Now, we are able to prove an upper and lower bound of the defined solution.
	\begin{lemma}[Upper bound]
		\label{lemma:upper_bound}
		Let the assumptions from Section~\ref{sec:transport_assump} be valid. Then, a solution $\rho$ of the initial-boundary-value problem \eqref{eq:transport_equation} and its trace $\gamma\rho$ are  componentwise bounded, i.e.\ it holds for all $t\in[0,T]$
		\begin{align}
		\label{eq:upper_bound_rho}
		0&\leq \rho(t,\cdot)\leq \left(\rho_{\max}+F_+(t)\right)\exp\left(\int_0^t\alpha(s)\mathrm ds\right)\\
			\shortintertext{almost everywhere in $\Omega$ and}
			\label{eq:upper_bound_gammarho}
			0&\leq \gamma\rho(t,\omega)\leq \left(\rho_{\max}+F_{+}(t)\right)\exp\left(\int_0^t\alpha(s)\mathrm ds\right)
		\end{align}
		for $|\mathrm d\mu_{u}|$-almost all $(t,\omega)\in \Gamma_T$. Here, it is $\alpha(t)=\left\|\left(u_x(t,\cdot)+c(t,\cdot)\right)^-\right\|_{L^{\infty}(\Omega)^n}$ and \begin{equation*}
		F_{+}(t)=\int_0^t\exp\left(-\int_0^s\alpha(r)\mathrm dr\right)\begin{pmatrix}
		\|(f_1(s,\cdot))^{+}\|_{L^{\infty}}\\
		\vdots\\
		\|(f_n(s,\cdot))^{+}\|_{L^{\infty}}
		\end{pmatrix}\mathrm ds.
		\end{equation*}
	\end{lemma}
	\begin{proof}
		Let  $\varphi\in C^{0,1}([0,T]\times\overline\Omega,\mathbb R^n)$ be an arbitrary test function with $\varphi(T)=0$. In order to prove that $r\rho$ with $r(t)=\exp\left(-\int_0^t\alpha(s)\mathrm ds\right)$ solves also a transport equation, we would like to use $\varphi r$ as a test function. Since $r$ is only absolutely continuous, but not necessarily Lipschitz continuous, we need to approximate $r$ by smoother functions. Therefore, let  $\alpha_k\in C^1([0,T])$ be a sequence converging to $\alpha$ in $L^1((0,T))$. We define
		\begin{equation*}
		r_k(t)=\exp\left(-\int_0^t\alpha_k(s)\mathrm ds\right).
		\end{equation*}
		The mean value theorem states for all $a\leq b\in \mathbb R$ the existence of $\zeta\in [a,b]$ with
		\begin{equation*}
		\exp(b)-\exp(a)=\exp(\zeta)(b-a).
		\end{equation*} 
		Thus, we find for each $t\in [0,T]$ a constant $\zeta(t)$ with 
		\begin{equation*}
		|\zeta(t)|\leq C_k:=\max\left(\|\alpha_k\|_{L^1((0,T))},\|\alpha\|_{L^1((0,T))}\right)
		\end{equation*} such that it holds
		\begin{align*}
		|r_k(t)-r(t)|&=\exp(\zeta(t))\left|\int_0^t\alpha_k(s)-\alpha(s)\mathrm ds\right|\\
		&\leq \exp(C_k)\|\alpha_k-\alpha\|_{L^1((0,T))}.
		\end{align*}
		The sequence $C_k$ is bounded as $\|\alpha_k\|_{L^1((0,T))}$ is convergent, thus $r_k\to r$ in $L^{\infty}((0,T))$. Especially, it also holds $\alpha_kr_k\to\alpha r$ in $L^1((0,T))$. Now, we will use $\varphi r_k$ as test functions in the weak formulation of the transport equation and take the limit for $k\to\infty$. Because of $(r_k)_t=-\alpha_kr_k$ this yields 
		{\allowdisplaybreaks\begin{align*}
		0&=\int_0^T\int_{\Omega}r_k\rho\transpose\left(\varphi_t+\mathbf U\varphi_x-\left(\mathbf C+\alpha_k\Id\right)\varphi\right)+r_kf\transpose\varphi\mathrm dx\mathrm dt\\*
		&\quad-\int_0^T\int_{\Gamma}r_k\gamma\rho\transpose(\nu\mathbf U)\varphi\mathrm d\omega\mathrm dt+\int_{\Omega}\rho_0\transpose\varphi(0)\mathrm dx\\
		&\to \int_0^T\int_{\Omega}r\rho\transpose\left(\varphi_t+\mathbf U\varphi_x-\left(\mathbf C+\alpha\Id\right)\varphi\right)+rf\transpose\varphi\mathrm dx\mathrm dt\\*
		&\quad-\int_0^T\int_{\Gamma}r\gamma\rho\transpose(\nu \mathbf U)\varphi\mathrm d\omega\mathrm dt+\int_{\Omega}\rho_0\transpose\varphi(0)\mathrm dx\\
		&=\int_0^T\int_{\Omega}\left(r\rho-F_+\right)\transpose\left(\varphi_t+\mathbf U\varphi_x-\left(\mathbf C+\alpha\Id\right)\varphi\right)\mathrm dx\mathrm dt\\*
		&\quad+\int_0^T\int_{\Omega}\left(rf-(F_+)_t\right)\transpose\varphi-F_+\transpose\left(\mathbf U_x+\mathbf C+\alpha\Id\right)\varphi\mathrm dx\mathrm dt\\*
		&\quad-\int_0^T\int_{\Gamma}\left(r\gamma\rho-F_+\right)\transpose(\nu \mathbf U)\varphi\mathrm d\omega\mathrm dt+\int_{\Omega}\rho_0\transpose\varphi(0)\mathrm dx.
		\end{align*}}
		Keeping in mind Remark~\ref{remark:f_L1}, this shows that $\bar \rho(t,x)=r(t)\rho(t,x)-F_+(t)$ with its trace $\gamma\bar{\rho}(t,\omega)=r(t)\gamma\rho(t,\omega)-F_+(t)$ is the unique solution of the following initial-boundary-value problem:
		\begin{align*}
			\bar{\rho}_t+(\mathbf U\bar{\rho})_x+\bar{\mathbf C}\bar \rho&=\bar f&\text{in }(0,T)\times\Omega\\
			\bar{\rho}(0,x)&=\rho_0(x)&\text{in }\Omega\\
			(\nu \mathbf U)^-\bar{\rho}&=(\nu \mathbf U)^-\bar{\mathcal H}(\bar\rho|_{\Gamma_T})&\mathrm d\mu_u^-\text{on }\Gamma_T
		\end{align*}
		with reaction term $\bar c=c+\alpha\Eins$, boundary operator $\bar{\mathcal H}(\rho)= r\rho_{\inn}-F_++\mathcal G(F_+)+\mathcal G(\rho)$ and source term $\bar f_i=rf_i-\|(rf_i)^+\|_{L^{\infty}(\Omega)}-\left((u_i)_x+\bar{c}_i\right)(F_+)_i$. For the boundary operator $\bar{\mathcal H}$ this conclusion needs further explanations. Since $\rho$ is a solution of \eqref{eq:transport_equation} we find
		\begin{align*}
		\gamma\bar{\rho}|_{\Gamma_T}&=r\gamma\rho|_{\Gamma_T}-F_+\\
		&=r\mathcal H(\gamma\rho|_{\Gamma_T})-F_+\\
		&=r\rho_{\inn}+\mathcal G(r\gamma\rho|_{\Gamma_T})-F_+\\
		&=r\rho_{\inn}-F_++\mathcal G(F_+)+\mathcal G(\gamma\bar\rho|_{\Gamma_T})\\
		&=\bar{\mathcal H}(\gamma\bar{\rho}|_{\Gamma_T}).
		\end{align*}
		The advantage of introducing $\bar \rho$ is that proving the upper bound \eqref{eq:upper_bound_rho} reduces to proving $\bar \rho\leq \rho_{\max}$, where $\rho_{\max}$ is known from the assumptions in Section~\ref{sec:transport_assump}. To this end, we define  $\beta\colon\mathbb R^n\to\mathbb R^n$ as
		\begin{equation*}
			\beta(s)=(s-\rho_{\max})^+
		\end{equation*}
		and we will show $\beta(\bar \rho)=0$ almost everywhere.
		
		For the test function $\varphi=\overline{\mathbf W}\Eins$ and all $t_0\in [0,T]$, the renormalization property yields
		\begin{equation}
		\begin{split}
		\label{eq:renormalization_max}
		&-\int_0^{t_0}\int_{\Omega}\Eins\transpose\overline{\mathbf W}\left(\bar{\mathbf C}D\beta(\bar \rho)\bar{\rho}-D\beta(\bar{\rho})\bar f\right)\mathrm dx\mathrm dt-\int_0^{t_0}\int_{\Omega}\Eins\transpose\overline{\mathbf W}\mathbf U_x(D\beta(\bar \rho)\bar \rho-\beta(\bar \rho))\mathrm dx\mathrm dt\\
		&-\int_0^{t_0}\int_{\Gamma}\beta(\gamma\bar \rho)\transpose\overline{\mathbf W}(\nu u)\mathrm d\omega\mathrm dt-\int_{\Omega}\Eins\transpose\overline{\mathbf W}\beta(\bar \rho(t_0))\mathrm dx=0
		\end{split}
		\end{equation}
		because of $\beta(\rho_0)=0$. The boundary term of this equation is non-positive. This can be shown using the componentwise monotonicity of $\beta$, the assumption \eqref{eq:exist_M} and the inequalities \eqref{eq:observation} and \eqref{obs:third}:
		\begin{align*}
			&\int_0^{t_0}\int_{\Gamma}\beta(\bar{\mathcal H}(\gamma\bar \rho))\transpose\overline{\mathbf W}(\nu u)^-\mathrm d\omega\mathrm dt\\&=\int_0^{t_0}\int_{\Gamma}\beta\left(r\rho_{\inn}-F_++\mathcal G(F_+)+\mathcal G(\gamma\bar \rho)\right)\transpose\overline{\mathbf W}(\nu u)^-\mathrm d\omega\mathrm dt\\
			&=\int_0^{t_0}\int_{\Gamma}\beta\left(r\rho_{\inn}-F_++\mathcal G(F_+)+\mathcal G(\rho_{\max})+\mathcal G(\gamma\bar \rho-\rho_{\max})\right)\transpose\overline{\mathbf W}(\nu u)^-\mathrm d\omega\mathrm dt\\
			&\leq \int_0^{t_0}\int_{\Gamma}\beta\left(\rho_{\max}+\mathcal G(\gamma\bar \rho-\rho_{\max})\right)\transpose\overline{\mathbf W}(\nu u)^-\mathrm d\omega\mathrm dt\\
			&= \int_0^{t_0}\int_{\Gamma}\left(\Bigl(\mathcal G(\gamma\bar \rho-\rho_{\max})\Bigr)^+\right)\transpose\overline{\mathbf W}(\nu u)^-\mathrm d\omega\mathrm dt\\
			&\leq\int_0^{t_0}\int_{\Gamma}\mathcal G\left(\left(\gamma\bar \rho-\rho_{\max}\right)^+\right)\transpose\overline{\mathbf W}(\nu u)^-\mathrm d\omega\mathrm dt\\
			&\leq\int_{0}^{t_0}\int_{\Gamma}\left(\left(\gamma\bar \rho-\rho_{\max}\right)^+\right)\transpose\overline{\mathbf W}(\nu u)^+\mathrm d\omega\mathrm dt\\
			&=\int_0^{t_0}\int_{\Gamma}\beta(\gamma\bar \rho)\transpose\overline{\mathbf W}(\nu u)^+\mathrm d\omega\mathrm dt.
		\end{align*}
		Because of $\beta(s)\geq 0$, $D\beta(s)s\geq 0$, $\bar c+ u_x=c+u_x+\|(c+u_x)^-\|_{L^{\infty}(\Omega)^n}\Eins\geq 0$ and $(D(\beta)\bar f)_i=D(\beta)_{ii}\left(rf_i-\|(rf_i)^+\|_{L^{\infty}(\Omega)}-((u_i)_x+\bar c_i)(F_+)_i\right)\leq 0$ we conclude from equation \eqref{eq:renormalization_max} that it holds
			\begin{equation}
			\begin{split}
			\label{eq:bound_ineq}
				\int_{\Omega}\Eins\transpose\overline{\mathbf W}\beta(\bar{\rho}(t_0))\mathrm dx
				&=-\int_0^{t_0}\int_{\Omega}\Eins\transpose\overline{\mathbf W}\left(\bar{\mathbf C}+\mathbf U_x\right)D\beta(\bar \rho)\bar \rho\mathrm dx\mathrm dt\\
				&\enskip+\int_0^{t_0}\int_{\Omega}\Eins\transpose\overline{\mathbf W}D\beta(\bar{\rho})\bar f\mathrm dx\mathrm dt+\int_0^{t_0}\int_{\Omega}\Eins\transpose\overline{\mathbf W}\mathbf U_x\beta(\bar \rho)\mathrm dx\mathrm dt\\
				&\enskip-\int_0^{t_0}\int_{\Gamma}\beta(\gamma\bar \rho)\transpose\overline{\mathbf W}(\nu u)^+\mathrm d\omega\mathrm dt +\int_0^{t_0}\int_{\Gamma}\beta(\bar{\mathcal H}(\gamma\bar \rho))\transpose\overline{\mathbf W}(\nu u)^-\mathrm d\omega\mathrm dt\\
				&\leq\int_0^{t_0}\|(u_x)^+\|_{L^{\infty}(\Omega)^n}\int_{\Omega}\Eins\transpose\overline{\mathbf W}\beta(\bar \rho)\mathrm dx\mathrm dt.
			\end{split}
			\end{equation}
		With Gronwall's inequality this leads to $\int_{\Omega}\Eins\transpose\overline{\mathbf W}\beta(\bar \rho(t_0))\mathrm dx=0$ and thus $\beta(\bar \rho(t_0))=0$ and $\bar \rho(t_0)\leq \rho_{\max}$ almost everywhere. Using the renormalization property and the uniqueness of the trace of $\beta(\bar \rho)=0$, we also conclude
		\begin{equation*}
			0=\gamma\beta(\bar{\rho})=\beta(\gamma\bar \rho)
		\end{equation*}
		and thus $\gamma\bar\rho(t,\omega)\leq \rho_{\max}$ holds for $|\mathrm d\mu_u|$-almost all $(t,\omega)$. Addition of $F$ and multiplication by $\exp\left(\int_0^t\alpha(s)\mathrm ds\right)$ shows the desired upper bound for $\rho$ and $\gamma\rho$.
		
		In order to prove the lower bound, we use the same procedure for $\bar\rho=r\rho$ and $\beta(s)=s^-$. In this case, it is $\bar c=c+\alpha\Eins$, $\bar f_i=rf_i$ and $\bar {\mathcal H}(\rho)=r\rho_{\inn}+\mathcal G(\rho)$. The only step we have to take into account again is the transition from equation \eqref{eq:renormalization_max} to inequality \eqref{eq:bound_ineq}. With the monotonicity of $\beta$, the positivity of $\rho_{\inn}$  and the inequalities \eqref{eq:observation} and \eqref{osb:third_a}, it follows
		\begin{align*}
			&\int_0^{t_0}\int_{\Gamma}\beta(\bar{\mathcal H}(\gamma\bar \rho))\transpose\overline{\mathbf W}(\nu u)^-\mathrm d\omega\mathrm dt\\
			&=\int_0^{t_0}\int_{\Gamma}\beta\left(r\rho_{\inn}+\mathcal G(\gamma\bar \rho)\right)\transpose\overline{\mathbf W}(\nu u)^-\mathrm d\omega\mathrm dt\\
			&=\int_0^{t_0}\left(\Big(r\rho_{\inn}+\mathcal G(\gamma\bar \rho)\Big)^-\right)\transpose\overline{\mathbf W}(\nu u)^-\mathrm d\omega\mathrm dt\\
			&\leq \int_0^{t_0}\int_{\Gamma}\left(\big(\mathcal G(\gamma\bar \rho)\big)^-\right)\transpose\overline{\mathbf W}(\nu u)^-\mathrm d\omega\mathrm dt\\
			&\leq\int_0^{t_0}\int_{\Gamma}\mathcal G\Big((\gamma\bar \rho)^-\Big)\transpose\overline{\mathbf W}(\nu u)^-\mathrm d\omega\mathrm dt\\
			&\leq\int_{0}^{t_0}\int_{\Gamma}\left(\left(\gamma\bar \rho\right)^-\right)\transpose\overline{\mathbf W}(\nu u)^+\mathrm d\omega\mathrm dt\\
			&=\int_0^{t_0}\int_{\Gamma}\beta(\gamma\bar \rho)\transpose\overline{\mathbf W}(\nu u)^+\mathrm d\omega\mathrm dt.
		\end{align*}
		Therefore, we can conclude as before $\bar \rho(t_0)\geq 0$ almost everywhere in $\Omega$ and $\gamma\bar\rho(t,\omega)\geq 0$ for $|\mathrm d\mu_u|$-almost all $(t,\omega)\in \Gamma_T$. This completes the proof.
	\end{proof}
	Beside the upper bound we just proved a lower bound. This lower bound can be sharpened if we ask for an additional assumption on $\mathcal G$ and use the fact $\rho\geq 0$ almost everywhere.
	\begin{lemma}[Lower bound]
		\label{lemma:lower_bound}
		In addition to the requirements from Section~\ref{sec:transport_assump}, we assume $(c+u_x)\in L^1((0,T),L^{\infty}(\Omega)^n)$. Furthermore, let $\rho_{\min}\in \mathbb R^n_{>0}$ be a vector such that it holds
		\begin{equation}
		\label{eq:exist_m}
		\rho_{\inn}\exp\left(\int_0^t\zeta(s)\mathrm ds\right)+\mathcal G(\rho_{\min})\geq \rho_{\min}
		\end{equation}
		component-by component for $\mathrm d\mu_u^-$-almost all $(t,\omega)\in \Gamma_T$ and
		\begin{equation*}
			\rho_0(x)\geq \rho_{\min}
		\end{equation*}
		for almost all $x\in \Omega$ with $\zeta(s)=\|(u_x(s,\cdot)+c(s,\cdot))^+\|_{L^{\infty}(\Omega)^n}$. Then, any solution of the initial-boundary-value problem \eqref{eq:transport_equation} is bounded from below. More precisely, it is 
		\begin{align*}
			\rho(t,\cdot)\geq \rho_{\min}\exp\left(-\int_0^t\zeta(s)\mathrm ds\right)>0\\
			\shortintertext{almost everywhere in $\Omega$ and}
			\gamma\rho(t,\omega)\geq \rho_{\min}\exp\left(-\int_0^t\zeta(s)\mathrm ds\right)>0
		\end{align*}
		for all $t\in[0,T]$ and $|\mathrm d\mu_{u}|$-almost all $(t,\omega)\in \Gamma_T$.
	\end{lemma}
	\begin{proof}
	Keeping in mind the non-negativity of $\rho$, this lemma can be proven in exactly the same manner as Lemma~\ref{lemma:upper_bound}.
	\end{proof}
	\section{Continuous Dependence on the data}
	\label{sec:cont_depend}
	The previous statements, especially the renormalization property of Theorem~\ref{theo:trace}, provide us with the necessary tools to prove a central result of this article: a kind of sequential continuity of the solution operator. We formulate this statement as a theorem with two parts. In the first part, the weak-$\star$ convergence of a sequence of solutions to a solution of the limit problem is obtained for very weak assumptions. This result will play a crucial role in the proof of existence of solutions of the transport equation, whereas the second part, proving the convergence in $C([0,T],L^p(\Omega)^n)$ under more restrictive assumptions, can for example be used for the proof of the existence of solutions of the low Mach number equations on a network (see \cite{RoggensackDiss.2014}).
	
	For the proof of the second part, we need the following result from the theory of Banach spaces concerning the uniform convergence of sequences of Banach valued functions. A proof can be found in \cite{RoggensackRadonRiesz.2014}.
	\begin{theorem}\label{lemma:gleichmaessig}
		Let $V$ be a uniformly convex and uniformly smooth Banach space. Let $(f_n)_n\subset C([0,T],V)$ be a sequence and $f\in C([0,T],V)$ a function. Let the pair $((f_n)_n,f)$ 
		fulfil the following two properties:
		\begin{enumerate}
			\item $\|f_n(t)\|_V$ converges uniformly to $\|f(t)\|_V$ and
			\item $\langle\varphi(t),f_n(t)\rangle$ converges uniformly to $\langle\varphi(t),f(t)\rangle$ for all $\varphi \in C([0,T],V')$.
		\end{enumerate} 
		Then, $f_n$ converges  to $f$ in the norm of $C([0,T],V)$, i.e.\
		\begin{equation*}
		\lim\limits_{n\to \infty}\sup\limits_{t\in[0,T]}\|f_n(t)-f(t)\|_V=0.
		\end{equation*}
	\end{theorem}
	\begin{proof}
	See \cite{RoggensackRadonRiesz.2014}.
	\end{proof}
	
\begin{theorem}[Stability]
	\label{theo:sequent_cont}
	1.\ 
	For all $k\in \mathbb N$ let $u_k$, $c_k$, $f_k$, $\rho_{0,k}$, $\rho_{\inn,k}$ and $\mathcal G_k$ be defined as in Section~\ref{sec:transport_assump}. Assume there exists for each $k$ a solution
	\begin{equation*}
	\rho_k\in L^{\infty}((0,T)\times\Omega)^n
	\end{equation*}
	of the initial-boundary-value problem
	\begin{align*}
	(\rho_k)_t+(\mathbf U_k\rho_k)_x+\mathbf C_k\rho_k&=f_k&&\text{in }(0,T)\times\Omega\\
	\rho_k(0)&=\rho_{0,k}&&\text{in }\Omega\\
	(\nu\mathbf U_k)^-\rho_k&=(\nu\mathbf U_k)^-\mathcal H_k(\rho_k|_{\Gamma_{T}})&&\text{on }\Gamma_T
	\end{align*}
	with trace $\gamma\rho_k\in L^{\infty}(\Gamma_{T},|\mathrm d\mu_{u_k}|)$.
	
	Moreover, we assume:
	\begin{description}
		\item[{uniform boundedness}]
		The sequence $(\rho_{\max,_k})_k\subset \mathbb R^n_{\geq 0}$ from assumption \eqref{eq:exist_M} is bounded, i.e.\  there is $\rho_{\max}\in \mathbb R^n_{\geq 0}$ with $\rho_{\max,k}\leq \rho_{\max}$.
		\item[{convergence of reaction terms}]
		The sequence $(c_k)_k$ strongly converges  to $c\in L^1((0,T)\times\Omega)^n$ in $L^1((0,T)\times\Omega)^n$.
		\item[{convergence of the velocities}] 
		$(u_k)_k$ strongly converges  in $L^1((0,T)\times\Omega)^n$ to $u\in L^1((0,T),W^{1,1}(\Omega)^n)$ with $(u_x+c)^-\in L^1((0,T),L^{\infty}(\Omega)^n)$ and $(u_x)^+\in L^1((0,T),L^{\infty}(\Omega)^n)$.
		\item[{convergence of velocities at the boundary}] 
		$(\nu u_k)_k$ strongly converges  to $\nu u$ in $L^{1}(\Gamma_T)^n$.
		\item[{boundedness\;\! of\;\! reaction\;\! terms\;\! and\;\! velocity\;\! derivatives}]
		The sequence\, $((c_k+(u_k)_x)^-)_k$ is bounded in $L^1((0,T),L^{\infty}(\Omega))$.
		\item[{weak convergence and boundedness of source terms}]
		The sequence $(f_k)_k$ weakly con\-ver\-ges  to $f\in L^1((0,T),L^{\infty}(\Omega)^n)$ in $L^1((0,T)\times\Omega)^n$ and it is bounded in $L^1((0,T),L^{\infty}(\Omega)^n)$.
		\item[{weak-$\star$ convergence of the initial conditions}] The sequence
		$(\rho_{0,k})_k$  weakly\nobreakdash-$\star$ converges to $\rho_{0}\in L^{\infty}(\Omega)^n$ in $L^{\infty}(\Omega)^n$.
		\item[{weak-$\star$ convergence of the inflow densities}]
		The sequence $(\rho_{\inn,k})_k$  weakly\nobreakdash-$\star$ converges to $\rho_{\inn}\in L^{\infty}(\Gamma_T,\mathrm d\mu_u^-)$ in $L^{\infty}(\Gamma_T,\mathrm d\mu_u^-)$.
		\item[{weak-$\star$ convergence/continuity of the boundary operators}]
		There exists an operator $\mathcal G$ fulfilling the assumptions of Section~\ref{sec:transport_assump} such that for each weak-\nobreakdash$\star$ convergent sequence $q_k\in L^{\infty}(\Gamma_T)^n$ with limit $q$ it holds
		\begin{equation*}
		(\nu \mathbf U_k)^-\mathcal G_k(q_k)(\nu u_k)^-\rightharpoonup(\nu \mathbf U)^-\mathcal G(q)
		\end{equation*}
		in $L^{1}(\Gamma_T)^n$.
	\end{description}
	Then, there exists a solution $\rho\in C([0,T],L^p(\Omega)^n)$ with trace $\gamma\rho\in L^{\infty}(\Gamma_T,|\mathrm d\mu_u|)$ of the transport equation
	\begin{align}
	\notag
	\rho_t+(\mathbf U)\rho_x+\mathbf C\rho&=f&&\text{in }(0,T)\times\Omega\\
	\label{eq:limit_prob}
	\rho(0)&=\rho_0&&\text{in }\Omega\\
	\notag
	(\nu\mathbf U)^-\rho &=(\nu\mathbf U)^-\mathcal H(\rho|_{\Gamma_T})&&\text{on }\Gamma_T.
	\end{align}
		Furthermore, it holds
		\begin{align*}
		\rho_k&\xrightharpoonup{\star}\rho&&\text{in }L^{\infty}((0,T)\times\Omega)^n,\\
		\gamma\rho_k&\xrightharpoonup{\star}\gamma\rho&&\text{in }L^{\infty}(\Gamma_T,|\mathrm d\mu_u|)
		\shortintertext{and}
		\rho_k(t)&\rightharpoonup\rho(t)&&\text{in }L^p(\Omega)^n\\
		\end{align*}
		 for all $t\in [0,T]$ and $p\in [1,\infty)$.
		 
		 2.\ We assume additionally:
		 \begin{description}
		 	\item[{convergence of the velocity derivative}]
		 	The sequence $\left((u_k)_x\right)_k$ strongly converges  to $u_x$ in the $L^1((0,T)\times\Omega)^n$-norm.
		 	\item[{convergence of reaction terms and velocity derivatives}]
			 $(\bar \alpha_k)_k$ with $\bar \alpha_k=\left\|(2c_k+(u_k)_x)^-\right\|_{L^{\infty}(\Omega)^n}$ strongly converges  to $\bar \alpha=\left\|(2c+u_x)^-\right\|_{L^{\infty}(\Omega)^n}$ in the  $L^1(0,T)$-norm.
		 	\item[\textbf{convergence of source terms}] The sequence
		 	$(f_k)_k$ strongly converges to $f$ in $L^1((0,T)\times\Omega)^n$.
			\item[{convergence of initial conditions}]
		 	The sequence $(\rho_{0,k})_k$ strongly converges to $\rho_{0}$ in the $L^1(\Omega)^n$-norm.
		 	\item[{weak lower semi-continuity of the boundary operators}]
		 	For each weak\nobreakdash-$\star$ convergent sequence $\rho_k\xrightharpoonup{\star}\rho$ in $L^{\infty}(\Gamma_T)^n$ it holds
		 	\begin{equation}
		 	\label{eq:lower_semi_cont}
		 	\begin{split}
		 	&\int_0^T\int_{\Gamma}\beta(\rho)^T\overline{\mathbf W}(\nu u)^+-\beta(\bar{\mathcal H}(\rho))^T\overline{\mathbf W}(\nu u)^-\mathrm d\omega\mathrm dt\\
		 	&\quad\leq\liminf\limits_k\int_0^T\int_{\Gamma}\beta(\rho_k)^T\overline{\mathbf W}(\nu u_k)^+-\beta(\bar{\mathcal H}_k(\rho_k))^T\overline{\mathbf W}(\nu u_k)^-\mathrm d\omega\mathrm dt
		 	\end{split}
		 	\end{equation}
		 	with $\beta\colon\mathbb R^n\to \mathbb R^n$ defined by $\beta_j(s)=s_j^{2}$, $\bar{\mathcal H}_k(\rho)=r_k\rho_{\inn,k}+\mathcal G_k(\rho)$ and $\bar{\mathcal H}(\rho)=r\rho_{\inn}+\mathcal G(\rho)$. Here, $r_k$ and $r$ are given by $r_k(t)=\exp\left(-\frac 12\int_0^t\bar\alpha_k(s)\mathrm ds\right)$ and $r(t)=\exp\left(-\frac 12\int_0^t\bar \alpha(s)\mathrm ds\right)$ and $\overline{\mathbf W}$ is the weight matrix introduced in Section~\ref{sec:transport_assump}.
		 \end{description}
		 Then, the convergence is even stronger:
		 \begin{align*}
		 \rho_k&\to\rho&&\text{in }C([0,T],L^p(\Omega)^n)\\
		 \intertext{and}
		 \gamma\rho_k&\to\gamma\rho&&\text{in }L^p(\Gamma_T,|\mathrm d\mu_u|)
		 \end{align*}
		 for all $p\in[1,\infty)$.
\end{theorem}
	\begin{proof}
			As before, we extend $\gamma\rho_k$, $\rho_{\inn,k}$, $\mathcal G_k(\gamma\rho_k)$ and $\mathcal H(\gamma\rho_k)$ by zero to functions in $L^{\infty}(\Gamma_T)^n$.
			
			1.\ First, we want to prove the existence of a solution of the limit problem and the weak-$\star$ convergence. Because of Lemma~\ref{lemma:upper_bound} and the boundedness of $\rho_{\max_k}\leq \rho_{\max}$, $\|(c_k+(u_k)_x)^-\|_{L^1((0,T),L^{\infty}(\Omega)^n)}\leq C_1$ and $\|f_k\|_{L^1((0,T),L^{\infty}(\Omega)^n)}\leq C_2$ the sequences $\rho_k$ and $\gamma\rho_k$ are bounded almost everywhere by
			\begin{align}
			&\max(\|\rho_k\|_{L^{\infty}((0,T)\times\Omega)^n},\|\gamma\rho_k\|_{L^{\infty}(\Gamma_T)^n})\notag\\
			&\leq \max\left((\rho_{\max,k}+F_k(t),F_k(T)+F_k(t)\right)\exp\left(\|(c_k+(u_k)_x)^-\|_{L^1((0,T),L^{\infty}(\Omega)^n)}\right)\notag\\
			\label{eq:upper_bound_uniform}
			&\leq \left(\max(\rho_{\max},C_2)+C_2\right)\exp(C_1)\\
			&=\bar\rho_{\max}.\notag
			\end{align}
			Here, it is \begin{equation*}
			F_{k}(t)=\int_0^t\exp\left(-\int_0^s\|(c_k+(u_k)_x)^-\|\mathrm dr\right)\begin{pmatrix}
			\|((f_k)_1(s,\cdot))^+\|_{L^{\infty}(\Omega)}\\
			\vdots\\
			\|((f_k)_n(s,\cdot))^+\|_{L^{\infty}(\Omega)}.
			\end{pmatrix}\mathrm ds
			\end{equation*}Thus, we can extract weak-$\star$ convergent subsequences - also denoted by $\rho_k$ and $\gamma\rho_k$ - with
			\begin{align*}
			\rho_k&\xrightharpoonup{\star}\rho&&\text{in }L^{\infty}((0,T)\times\Omega)^n\\
			\shortintertext{and}
			\gamma\rho_k&\xrightharpoonup{\star}q&&\text{in }L^{\infty}(\Gamma_T)^n.
			\end{align*}
			As a second step, we will show that $\rho$ is a solution of the transport equation and that $q$ is its trace and fulfils the boundary conditions.
			
			Because of the strong convergence of $u_k$ and $c_k$ in $L^1((0,T)\times\Omega)^n$, the sequences $\mathbf U_k\rho_k$ and $\mathbf C_k\rho_k$ converge weakly in $L^1((0,T)\times\Omega)^n$, i.e.\ 
			\begin{align*}
			\mathbf U_k\rho_k&\rightharpoonup\mathbf U\rho\\
			\shortintertext{and}
			\mathbf C_k\rho_k&\rightharpoonup\mathbf C\rho.
			\end{align*}
			Similarly, $(\nu\mathbf U_k)\gamma\rho_k$ converges weakly to $(\nu\mathbf U)\rho$ in $L^1(\Gamma_T)^n$.
			Thus, for a test function $\varphi\in C^{0,1}([0,T]\times\overline\Omega)$ with $\varphi(T)=0$ it holds
			\begin{align*}
			0&=\int_0^T\int_{\Omega}\rho_k\transpose\left(\varphi_t+\mathbf U_k\varphi_x-\mathbf C_k\varphi\right)+f_k\transpose\varphi\mathrm dx\mathrm dt\\
			&\quad+\int_{\Omega}\rho_{0,k}\transpose\varphi(0)\mathrm dx-\int_0^T\int_{\Gamma}\gamma\rho_k\transpose(\nu\mathbf U_k)\varphi\mathrm d\omega\mathrm dt\\
			&\to \int_0^T\int_{\Omega}\rho\transpose\left(\varphi_t+\mathbf U\varphi_x-\mathbf C\varphi\right)+f\transpose\varphi\mathrm dx\mathrm dt\\
			&\quad+\int_{\Omega}\rho_0\transpose\varphi(0)\mathrm dx-\int_0^T\int_{\Gamma}q\transpose(\nu\mathbf U)\varphi\mathrm d\omega\mathrm dt.
			\end{align*}
			In words, this means that $\rho$ is a solution of the transport equation with its unique trace $\gamma \rho=q$. Furthermore, because of Theorem~\ref{theo:trace} it is $\rho\in C([0,T],L^p(\Omega)^n)$ and $\rho$ fulfils the initial condition $\rho(0)=\rho_0$. For the boundary term we have to consider the trace $q$ in more detail. The assumptions on $\mathcal G_k$ and $\rho_{\inn,k}$ yield the weak convergence of $(\nu\mathbf U_k)^-\rho_{\inn,k}$ and $(\nu\mathbf U_k)^-\mathcal G_k(\gamma\rho_k)$. Thus,
			\begin{align*}
			\int_0^T\int_{\Gamma}q\transpose(\nu\mathbf U)^-\varphi\mathrm d\omega\mathrm dt&\gets\int_{0}^{T}\int_{\Gamma}\gamma\rho_k\transpose(\nu\mathbf U_k)^-\varphi\mathrm d\omega\mathrm dt\\
			&=\int_0^T\int_{\Gamma}\mathcal H(\gamma\rho_k)\transpose(\nu\mathbf U_k)^-\varphi\mathrm d\omega\mathrm dt\\
			&=\int_0^T\int_{\Gamma}\left(\rho_{\inn,k}+\mathcal G_k(\gamma\rho_k)\right)\transpose(\nu\mathbf U_k)^-\varphi\mathrm d\omega\mathrm dt\\
			&\to\int_0^T\int_{\Gamma}\left(\rho_{\inn}+\mathcal G(q)\right)\transpose(\nu\mathbf U)^-\varphi\mathrm d\omega\mathrm dt\\
			&=\int_0^T\int_{\Gamma}\mathcal H(q)\transpose(\nu\mathbf U)^-\varphi\mathrm d\omega\mathrm dt
			\end{align*}
			is true and the trace $q=\gamma \rho$ fulfils the boundary condition
			\begin{align*}
			(\nu\mathbf U)^-\gamma q=(\nu\mathbf U)^-\mathcal H(\gamma q).
			\end{align*}
			Hence, $\rho$ is the unique solution of the limit problem \eqref{eq:limit_prob}. Since the weak-$\star$ limits $\rho$ and $q$ are unique, in fact the whole sequences and not only subsequences converge.
			
			As last step of the first statement, we have to show the weak convergence of $\rho_k(t)$ for each $t\in [0,T]$. However, this is rather simple, as  the sequence $\rho_k(t)$ is bounded in $L^p(\Omega)^n$ because of the continuity of $\rho_k$ with values in $L^p(\Omega)^n$ and Lemma~\ref{lemma:upper_bound}. Thus, there exists a weak convergent subsequence in $L^{p}(\Omega)^n$ with
			\begin{equation*}
			\rho_k(t)\rightharpoonup q.
			\end{equation*} For this subsequence and for all time-independent test functions $\varphi\in C^{0,1}(\overline \Omega,\mathbb R^n)$ it holds with Theorem~\ref{theo:trace}
			\begin{align*}
			\int_{\Omega}q\transpose\varphi\mathrm dx&\gets\int_{\Omega}\rho_k(t)\transpose\varphi\mathrm dx\\&=\int_0^t\int_{\Omega}\rho_k\transpose\left(\mathbf U_k\varphi_x-\mathbf C_k\varphi\right)+f_k\transpose\varphi\mathrm dx\mathrm dt\\
			&\quad+\int_{\Omega}\rho_{0,k}\transpose\varphi\mathrm dx-\int_0^t\int_{\Gamma}\gamma\rho_k\transpose(\nu\mathbf U_k)\varphi\mathrm d\omega\mathrm dt\\
			&\to \int_0^t\int_{\Omega}\rho\transpose\left(\mathbf U\varphi_x-\mathbf C\varphi\right)+f\transpose\varphi\mathrm dx\mathrm dt\\
			&\quad+\int_{\Omega}\rho_0\transpose\varphi\mathrm dx-\int_0^t\int_{\Gamma}\gamma\rho\transpose(\nu\mathbf U)\varphi\mathrm d\omega\mathrm dt\\
			&=\int_{\Omega}\rho(t)\transpose\varphi\mathrm dx.
			\end{align*}
			Thus, we conclude
			\begin{equation*}
			q=\rho(t)
			\end{equation*}
			due to the density of $C^{0,1}(\overline{\Omega},\mathbb R^n)$ in $L^{p'}(\Omega)^n$.
			Again, by the uniqueness of the solution $\rho$ we see that the whole sequence converges weakly.
			
			2.\ Now, we consider the stronger assumptions in order to prove both, the uniform convergence of $\rho_k$ with values in $L^p(\Omega)^n$ and the strong convergence of the trace. As a first step, we prove that the convergence of $\rho_k(t)$ is actually strong in $L^p(\Omega)^n$ for each $t\in [0,T]$. To this end, we use the Radon-Riesz property (see e.g.\ \cite{megginson1998}) of the $L^p$-spaces with $p\in(1,\infty)$, i.e.\  the fact
			\begin{equation}
			\label{eq:Radon-Riesz}
			\left.\begin{array}{l}
				f_k\rightharpoonup f\text{ in } L^{p}(\Omega)^n\\
				\|f_k\|_{L^p(\Omega)^n}\rightarrow \|f\|_{L^p(\Omega)^n}
			\end{array}\right\}
			\Rightarrow f_k\rightarrow f \text{ in } L^p(\Omega)^n.
			\end{equation}
		
			We define the auxiliary variable
			\begin{equation*}
			\bar \rho_k(t,x)=\rho_k(t,x)r_k(t)
			\end{equation*}
			with $r_k(t)=\exp\left(-\frac12\int_0^t\bar\alpha_k(s)\mathrm ds\right)$ and $\bar\alpha_k(s)=\left\|\left(2c_k+(u_k)_x\right)^-\right\|_{L^{\infty}(\Omega)^n}.$ By assumption, it holds $\bar{\alpha}_k\to \bar{\alpha}=\left\|\left(2c+u_x\right)^-\right\|_{L^{\infty}(\Omega)^n}$ in $L^1((0,T))$. Thus, we conclude
			\begin{equation*}
			r_k(t)\to r(t)=\exp\left(-\frac12\int_0^t\alpha(s)\mathrm ds\right)
			\end{equation*}in $C([0,T])$, as in the proof of Lemma~\ref{lemma:upper_bound}. We also see, as in that proof, that $\bar \rho_k$ solves the equation
			\begin{align*}
			(\bar\rho_k)_t+\left(\mathbf U_k\bar\rho_k\right)_x+\bar{\mathbf C}_k\bar\rho_k&=\bar f_k&&\text{in }(0,T)\times\Omega\\
			\bar\rho_k(0)&=\rho_{0,k}&&\text{in }\Omega\\
			(\nu\mathbf U_k)^-\bar\rho_k&=(\nu\mathbf U_k)^-\bar{\mathcal H}_k(\bar\rho_k|_{\Gamma_T})&&\text{on }\Gamma_T
			\end{align*}
			with $\bar c_k=c_k+\frac12\bar{\alpha}_k\Eins$, $\bar f_k=r_kf_k$ and $\bar{\mathcal H}_k(\bar \rho)=r_k(t)\rho_{\inn,k}+\mathcal G_k(\bar{\rho})$. For $\bar \rho_k$ and this problem, all requirements of the first part of this theorem are fulfilled. Thus, we conclude
			\begin{align*}
			\bar \rho_k(t)\rightharpoonup \bar\rho(t)
			\end{align*}
			in $L^p(\Omega)^n$, where $\bar \rho$ is the solution of
						\begin{align*}
						(\bar\rho)_t+\left(\mathbf U\bar\rho\right)_x+\bar{\mathbf C}\bar\rho&=\bar f&&\text{in }(0,T)\times\Omega\\
						\bar\rho(0)&=\rho_{0}&&\text{in }\Omega\\
						(\nu\mathbf U)^-\bar\rho_k&=(\nu\mathbf U)^-\bar{\mathcal H}(\bar\rho|_{\Gamma_T})&&\text{on }\Gamma_T
						\end{align*}
			with $\bar c=c+\frac12 \bar{\alpha}\Eins$, $\bar f=rf$ and $\bar{\mathcal H}(\bar \rho)=r(t)\rho_{\inn}+\mathcal G(\bar{\rho})$.
			
			The renormalization property for $\beta\colon \mathbb R^n\to\mathbb R^n$ defined by $\beta(s)_j=s_j^2$ and $\varphi=\overline{\mathbf W}\Eins$, the weak lower semi-continuity of the $L^2$-Norm, the strong convergence of $\rho_{0,k}$ and the weak convergence of $\bar f_k\transpose\overline{\mathbf W}\bar \rho_k$ yield
				{\allowdisplaybreaks\begin{align}
				&\int_{\Omega}\Eins\transpose\overline{\mathbf W}\beta(\rho_0)\mathrm dx-\int_0^t\int_{\Gamma}\beta(\gamma{\bar\rho})\transpose\overline{\mathbf W}(\nu u)\mathrm d\omega\mathrm dt\notag\\*
				&\quad-\int_0^t\int_{\Omega}(2\bar c+u_x)\transpose\overline{\mathbf W}\beta(\bar\rho)\mathrm dx\mathrm dt-2\bar f\transpose\overline{\mathbf W}\bar\rho\mathrm dx\mathrm dt\notag\\
				&=\int_{\Omega}\Eins\transpose\overline{\mathbf W}\beta(\bar\rho(t))\mathrm dx\notag\\
				&=\|\overline{\mathbf W}\bar\rho(t)\|_{L^{2}(\Omega)^n}^{2}\notag\\
				&\leq 
				\liminf_k\|\overline{\mathbf W}\bar\rho_{k}(t)\|_{L^{2}(\Omega)^n}^{2}\label{eq:liminf}\\
				&=\liminf_k\Biggl(\int_{\Omega}\Eins\transpose\overline{\mathbf W}\beta(\rho_{0,k})\mathrm dx-\int_0^t\int_{\Gamma}\beta({\gamma\bar\rho}_{k})\transpose\overline{\mathbf W}(u_{k}\nu)\mathrm d\omega\mathrm dt\notag\\*
				&\quad\quad-\int_0^t\int_{\Omega}((2\bar c_k+u_{k})_x)\transpose\overline{\mathbf W}\beta(\bar\rho_{k})-2\bar f_k\transpose\overline{\mathbf W}\bar\rho_k\mathrm dx\mathrm dt\Biggr)\notag\\
				&=\int_{\Omega}\Eins\transpose\overline{\mathbf W}\beta(\rho_0)\mathrm dx\notag\\*
				&\quad+\liminf_k\Biggl(\int_0^t\int_{\Gamma}\beta(\bar{\mathcal H}_{k}({\gamma\bar\rho}_{k}))\transpose\overline{\mathbf W}(u_{k}\nu)^--\beta(\overline{\gamma\rho}_{k})\transpose\overline{\mathbf W}(u_{k}\nu)^+\mathrm d\omega\mathrm dt\notag\\*
				&\quad\quad-\int_0^t\int_{\Omega}(2\bar c_k+(u_{k})_x)\transpose\overline{\mathbf W}\beta(\bar\rho_{k})\mathrm dx\mathrm dt\Biggr)+\int_0^t\int_{\Omega}2\bar f\transpose\overline{\mathbf W}\bar\rho\mathrm dx\mathrm dt\notag
				\end{align}}
				and thus
				\begin{equation}
				\label{eq:limsup}
				\begin{split}
				&\limsup_k\Biggl(\int_0^t\int_{\Gamma}\beta({\gamma\bar\rho}_{k})\transpose\overline{\mathbf W}(u_{k}\nu)^+-\beta(\bar{\mathcal H}_{k}({\gamma\bar\rho}_{k}))\transpose\overline{\mathbf W}(u_{k}\nu)^-\mathrm d\omega\mathrm dt\\
				&\quad\quad+\int_0^t\int_{\Omega}(2\bar c_k+(u_{k})_x)\transpose\overline{\mathbf W}\beta(\bar \rho_{k})\mathrm dx\mathrm dt\Biggr)\\
				&\leq \int_0^t\int_{\Gamma}\beta({\gamma\bar\rho})\transpose\overline{\mathbf W}(\nu u)^+-\beta(\bar{\mathcal H}({\gamma\bar\rho}))\transpose\overline{\mathbf W}(\nu u)^-\mathrm d\omega\mathrm dt\\
				&\quad+\int_0^t\int_{\Omega}(2\bar c+u_x)\transpose\overline{\mathbf W}\beta(\bar\rho)\mathrm dx\mathrm dt.
				\end{split}
				\end{equation}
				Using the assumption \eqref{eq:lower_semi_cont} for the weak-$\star$ convergent sequence $\chi_{[0,t]}{\gamma\bar\rho}_{k}$, we conclude 
				\begin{equation}
				\label{eq:inequality_H}
				\begin{split}
				&\int_0^t\int_{\Gamma}\beta({\gamma\bar\rho})^T\overline{\mathbf W}(\nu u)^+-\beta(\bar{\mathcal H}({\gamma\bar\rho}))^T\overline{\mathbf W}(\nu u)^-\mathrm d\omega\mathrm dt\\
				&\quad\leq\liminf\limits_k\int_0^t\int_{\Gamma}\beta({\gamma\bar\rho}_{k})^T\overline{\mathbf W}(\nu u_k)^+-\beta(\bar{\mathcal H}_k({\gamma\bar\rho}_{k}))^T\overline{\mathbf W}(\nu u_k)^-\mathrm d\omega\mathrm dt
				\end{split}
				\end{equation}
				for all $t\in [0,T]$. Because of the strong convergence of $c_k$ and $(u_k)_x$ in $L^1((0,T)\times\Omega)^n$, we have for $j=1,\ldots,n$ also the convergence 
						\begin{equation*}
						\left|2\bar c^j_k+(u^j_k)_x\right|^{\frac 1{2}}\rightarrow\left|2\bar c^j+(u^j)_x\right|^{\frac 1{2}}
						\end{equation*} in $L^{2}([0,T]\times\Omega)$ and hence also the weak convergence
						\begin{equation*}
						\left|2\bar c^j_k+(u^j_k)_x\right|^{\frac 1{2}}\bar \rho_k^j\rightharpoonup\left|2\bar c^j+(u^j)_x\right|^{\frac 1{2}}\bar \rho^j
						\end{equation*}
				in $L^{2}((0,T)\times\Omega)$. Here, the superscripts denote the components of the vectors. Since it holds
				$2\bar c_k+(u_k)_x=2c_k+(u_k)_x+\|(2c_k+(u_k)_x)^-\|_{L^{\infty}(\Omega)^n}\Eins\geq 0,$				
				 we can again use the weak lower semi-continuity of the norm to find
				\begin{align*}
				\int_0^t\int_{\Omega}(2\bar c+u_x)^T\overline{\mathbf W}\beta(\bar\rho)\mathrm dx\mathrm dt&=\sum_{j=1}^n\|(2\bar c^j+(u^j)_x)^{\frac 12}\mathbf  W_{jj}\bar\rho^j\|^2_{L^2((0,t)\times\Omega)}\\
				&\leq\liminf\limits_k\sum_{j=1}^n\|(2\bar c_k^j+(u_k^j)_x)^{\frac 12}\overline{\mathbf W}_{jj}\bar \rho^j_k\|^2_{L^2((0,t)\times\Omega)}\\
				&= \liminf\limits_k\int_0^t\int_{\Omega}\left(2\bar c_k+(u_k)_x\right)^T\overline{\mathbf W}\beta(\bar\rho_k)\mathrm dx\mathrm dt.
				\end{align*}
				This leads with the inequalities \eqref{eq:limsup} and \eqref{eq:inequality_H} to
				{\allowdisplaybreaks\begin{align*}
				&\limsup_k\Biggl(\int_0^t\int_{\Gamma}\beta({\gamma\bar\rho}_{k})\transpose\overline{\mathbf W}(u_{k}\nu)^+-\beta(\bar{\mathcal H}_{k}({\gamma\bar\rho}_{k}))\transpose\overline{\mathbf W}(u_{k}\nu)^-\mathrm d\omega\mathrm dt\\*
				&\quad\quad+\int_0^t\int_{\Omega}(2\bar c_k+(u_{k})_x)\transpose\overline{\mathbf W}\beta(\bar\rho_{k})\mathrm dx\mathrm dt\Biggr)\\
				&\leq \int_0^t\int_{\Gamma}\beta({\gamma\bar\rho})\transpose\overline{\mathbf W}(\nu u)^+-\beta(\bar{\mathcal H}({\gamma\bar\rho}))\transpose\overline{\mathbf W}(\nu u)^-\mathrm d\omega\mathrm dt\\*
				&\quad+\int_0^t\int_{\Omega}(2\bar c+u_x)\transpose\overline{\mathbf W}\beta(\bar\rho)\mathrm dx\mathrm dt\\
				&\leq \liminf_k\Biggl(\int_0^t\int_{\Gamma}\beta({\gamma\bar\rho}_{k})\transpose\overline{\mathbf W}(u_{k}\nu)^+-\beta(\bar{\mathcal H}_{k}({\gamma\bar\rho}_{k}))\transpose\overline{\mathbf W}(u_{k}\nu)^-\mathrm d\omega\Biggr)\mathrm dt\\*
				&\quad+\liminf_k\int_0^t \int_{\Omega}(2\bar c_k+(u_{k})_x)\transpose\overline{\mathbf W}\beta(\bar \rho_{k})\mathrm dx\mathrm dt\\
				&\leq \liminf_k\Biggl(\int_0^t\int_{\Gamma}\beta({\gamma\bar\rho}_{k})\transpose\overline{\mathbf W}(u_{k}\nu)^+-\beta(\bar{\mathcal H}_{k}({\gamma\bar\rho}_{k}))\transpose\overline{\mathbf W}(u_{k}\nu)^-\mathrm d\omega\mathrm dt\\*
				&\quad\quad+\int_0^t\int_{\Omega}(2\bar c_k+(u_{k})_x)\transpose\overline{\mathbf W}\beta(\bar\rho_{k})\mathrm dx\mathrm dt\Biggr).
				\end{align*}}
				Hence, the limit inferior and the limit superior coincide and consequently they are equal to the limit. Altogether, this yields
			
				{\allowdisplaybreaks	\begin{align*}
				&\lim\limits_{k\to\infty}\|\overline{\mathbf W}\bar\rho_k(t)\|_{L^2(\Omega)^n}^2\\
				&=\lim\limits_{k\to\infty}\Biggl(-\int_0^t\int_{\Gamma}\beta({\gamma\bar\rho}_{k})\transpose\overline{\mathbf W}(u_{k}\nu)^+-\beta(\bar{\mathcal H}_{k}({\gamma\bar\rho}_{k}))\transpose\overline{\mathbf W}(u_{k}\nu)^-\mathrm d\omega\mathrm dt\\*
				&\quad\quad+\int_{\Omega}\Eins\transpose\overline{\mathbf W}\beta(\rho_{0,k})\mathrm dx-\int_0^t\int_{\Omega}(2\bar c_k+(u_{k})_x)\transpose\overline{\mathbf W}\beta(\bar\rho_{k})-2\bar f_k\transpose\overline{\mathbf W}\bar\rho_k\mathrm dx\mathrm dt\Biggr)\\
				&=-\int_0^t\int_{\Gamma}\beta({\gamma\bar\rho})\transpose\overline{\mathbf W}(\nu u)^+-\beta(\bar{\mathcal H}({\gamma\bar\rho}))\transpose\overline{\mathbf W}(\nu u)^-\mathrm d\omega\mathrm dt\\*
				&\quad+\int_{\Omega}\Eins\transpose\overline{\mathbf W}\beta(\rho_0)\mathrm dx-\int_0^t\int_{\Omega}(2\bar c+u_x)\transpose\overline{\mathbf W}\beta(\bar\rho)-2\bar f\transpose\overline{\mathbf W}\bar\rho\mathrm dx\mathrm dt\\
				&=\|\overline{\mathbf W}\bar\rho(t)\|_{L^2(\Omega)^n}^2.
				\end{align*}}
				Now, the strong convergence of $\bar \rho_{k}(t)$ for all $t\in [0,T]$ follows by the Radon-Riesz property \eqref{eq:Radon-Riesz} of $L^{2}(\Omega)^n$ and the equivalence of the weighted Euclidean norm $\|\overline{\mathbf W}\cdot\|_2$ and the Euclidean norm $\|\cdot\|_2$. Thus, by the convergence of $\bar r_k$
				\begin{equation*}
				\bar r_k(t)=\exp\left(\frac 12\int_0^t\bar\alpha_k(s)\mathrm ds\right)\to\bar r(t)=\exp\left(\frac 12\int_0^t\bar{\alpha}(s)\mathrm ds\right)
				\end{equation*}
				in $C([0,T])$ (see proof of Lemma~\ref{lemma:upper_bound}), the convergence
				\begin{align*}
				\rho_k(t)=\bar{\rho_k}(t)\bar r_k(t)\to\bar \rho(t)\bar r(t)=\rho(t)
				\end{align*}
				in $L^2(\Omega)^n$ is also true.
				Because of the uniform boundedness of $\rho_k(t)$ in $L^{\infty}(\Omega)^n$, we see that $\rho_k(t)$ converges in all $L^p(\Omega)^n$-spaces for $p\in[1,\infty)$. 

Up to now, we have shown the pointwise in time convergence of $\rho_k$.
		In the next step, we want to show that the convergence of $\rho_k$ is uniform in $[0,T]$ with values in $L^p(\Omega)^n$. This can be done by applying Theorem~\ref{lemma:gleichmaessig} to $\rho_k$, a Radon-Riesz like property for uniform convergence. Therefore, we first prove that $\int_{\Omega}\rho_k(t,x)\transpose\varphi(t,x)\mathrm dx$ converges uniformly towards $\int_{\Omega}\rho(t,x)\transpose\varphi(t,x)\mathrm dx$ for all fixed $\varphi\in C^1([0,T]\times\overline{\Omega},\mathbb R^n)$.
		Recall the upper bound  $\bar \rho_{\max}\in \mathbb R_{>0}$ for $\rho_k$ from the first part of the proof and let be $\varepsilon>0$. We define $\tilde \varepsilon=\frac{\varepsilon}{5\bar\rho_{\max}\|\varphi\|_{C^1([0,T]\times\overline\Omega,\mathbb R^n)}}$. Since $u_k$, $c_k$, $f_k$ and $\nu u_k$ are converging in $L^1$, they are equi-integrable. Thus, there 
		exists $\delta>0$ such that $S\subset [0,T]$ with $\mu(S)<\delta$ implies the four inequalities: $\int_S\|u_k\|_{L^1(\Omega)^n}\mathrm dt<\tilde\varepsilon$, $\int_S\|c_k\|_{L^1(\Omega)^n}\mathrm dt<\tilde{\varepsilon}$, $\int_S\|f_k\|_{L^1(\Omega)^n}\mathrm dt<\tilde{\varepsilon}$ and $\int_S\|\nu u_k\|_{L^1(\Gamma)^n}\mathrm dt<\tilde \varepsilon$. Then, we compute for $t_1,t_2\in[0,T]$ with $|t_2-t_1|<\min(\delta,\tilde{\varepsilon})$
		{\allowdisplaybreaks
		\begin{align*}
		&\left|\int_{\Omega}\varphi(t_2)\transpose\rho_k(t_2)\mathrm dx-\int_{\Omega}\varphi(t_1)\transpose\rho_k(t_1)\mathrm dx\right|\\
		&=\left|\int_{t_1}^{t_2}\int_{\Omega}\left(\varphi_t\transpose+\varphi_x\transpose\mathbf U_k-\varphi\transpose\mathbf C_k\right)\rho_k+f_k\transpose\varphi\mathrm dx\mathrm dt-\int_{t_1}^{t_2}\int_{\Gamma}\varphi\transpose(\nu \mathbf U_k)\gamma\rho_k\mathrm d\omega\mathrm dt\right|\\
		&\leq\|\varphi\|_{C^1}\bar \rho_{\max}\int_{t_1}^{t_2}\left(1+\|u_k\|_{L^1(\Omega)^n}+\|c_k\|_{L^1(\Omega)^n}+\|f_k\|_{L^1(\Omega)^n}\right)\mathrm dt\\*
		&\quad+\|\varphi\|_{\infty}\bar\rho_{\max}\int_{t_1}^{t_2}\|\nu u_k\|_{L^1(\Gamma)^n}\mathrm dt\\
		&\leq \bar\rho_{\max}\|\varphi\|_{C^1}\int_{t_1}^{t_2}\left(1+\|u_k\|_{L^1(\Omega)^n}+\|c_k\|_{L^1(\Omega)^n}+\|f_k\|_{L^1(\Omega)^n}+\|\nu u_k\|_{L^1(\Gamma)^n}\right)\mathrm dt\\
		&<\varepsilon
		\end{align*}}
		which proves the weak equicontinuity of $\rho_k$. Together with the weak convergence of $\rho_k(t)$ in $L^p(\Omega)$, we conclude the uniform in time convergence of $\int_{\Omega}\rho_k(t)\transpose\varphi(t)\mathrm dx$ as each equicontinuous pointwise convergent sequence is uniformly convergent. Due to the density of $C^1([0,T]\times\overline \Omega)^n$ in $C([0,T],L^2(\Omega)^n)$, this uniform convergence holds in fact for all $\varphi\in C([0,T],L^{2}(\Omega)^n)$. 
		
		To apply Theorem~\ref{lemma:gleichmaessig}, we additionally need the uniform convergence of the norm $\|\rho_k(t)\|_{L^{2}(\Omega)^n}$. Therefore, let again be $\varepsilon>0$, define $\tilde \varepsilon=\frac{\varepsilon}{4\max(\bar \rho_{\max},1)\bar\rho_{\max}}$ and choose a value $\delta>0$ such that for all $S\subset [0,T]$ with $\mu(S)<\delta$ the inequalities $\int_S \|2c_k+(u_k)_x\|_{L^1(\Omega)^n}\mathrm dt<{\tilde\varepsilon}$, $\int_S \|f_k\|_{L^1(\Omega)^n}\mathrm dt<{\tilde\varepsilon}$ and $\int_S\|\nu u_k\|_{L^1(\Gamma)^n}\mathrm dt<\tilde\varepsilon$ hold. Then, the renormalization property leads to
		{\allowdisplaybreaks\begin{align*}
		&\left|\|\rho_k(t_2)\|_{L^{2}(\Omega)^n}^{2}-\|\rho_k(t_1)\|_{L^{2}(\Omega)^n}^{2}\right|\\
		&=\left|\int_{t_1}^{t_2}\int_{\Gamma}\beta(\gamma\rho_k)\transpose(\nu u_k)\mathrm d\omega\mathrm dt+\int_{t_1}^{t_2}\int_{\Omega}\left(2c_k+(u_k)_x\right)\transpose\beta(\rho_k)-2f_k\transpose\rho_k\mathrm dx\mathrm dt\right|\\
		&\leq \max(\bar\rho_{\max},1)\bar\rho_{\max}\int_{t_1}^{t_2}\left(\|\nu u_k\|_{L^1(\Gamma)^n}\mathrm+\|2c_k+(u_k)_x\|_{L^1(\Omega)^n}+2\|f_k\|_{L^1(\Omega)^n}\right)\mathrm dt\\
		&<\varepsilon
		\end{align*}
		for $t_1,t_2\in [0,T]$ with $|t_2-t_1|<\delta$.}
		Again, the shown equicontinuity yields together with the pointwise convergence of $\|\rho_k(t)\|_{L^{2}(\Omega)^n}$ the uniform convergence. As a consequence, we can apply Theorem~\ref{lemma:gleichmaessig} to deduce the convergence of $\rho_k$ in  $C([0,T],L^{2}(\Omega)^n)$. As before, the convergence holds in fact in all the spaces $C([0,T],L^p(\Omega)^n)$ for $p\in [1,\infty)$ because of the uniform boundedness.
		
		To finish the proof, it remains to show the strong convergence of the trace $\gamma\rho_k$. Again, 
		we consider the renormalization property for $\beta$ as before with an arbitrary $\varphi\in C^{0,1}([0,T]\times\overline{\Omega},\mathbb R^n)$. Since $\beta(\rho_{k})$ converges strongly to $\beta(\rho)$ in $L^p((0,T)\times\Omega)^n$ and $\beta(\rho_k(T))$ to $\beta(\rho(T))$ in $L^p(\Omega)^n$ we conclude 
		\begin{equation}
		\label{eq:conv_gamma_rho}
		\begin{split}
		&\int_0^T\int_{\Gamma}\beta(\gamma\rho_k)\transpose(\nu\mathbf U_k)\varphi\mathrm d\omega\mathrm dt\\
		&=\int_0^T\int_{\Omega}\beta(\rho_k)\transpose\left(\varphi_t+\mathbf U_k\varphi_x-(2\mathbf C_k
		+(\mathbf U_k)_x)\varphi\right)+2f_k\transpose\rho_k\mathrm dx\mathrm dt\\
		&\quad+\int_{\Omega}\beta(\rho_{0,k})\transpose\varphi(0)\mathrm dx-\int_{\Omega}\beta(\rho_k(T))\transpose\varphi(T)\mathrm dx\\
		&\to\int_0^T\int_{\Omega}\beta(\rho)\transpose\left(\varphi_t+\mathbf U\varphi_x-(2\mathbf C+\mathbf U_x)\varphi\right)+2f\transpose\rho\mathrm dx\mathrm dt\\
		&\quad+\int_{\Omega}\beta(\rho_0)\transpose\varphi(0)\mathrm dx-\int_{\Omega}\beta(\rho(T))\transpose\varphi(T)\mathrm dx\\
		&=\int_0^T\int_{\Gamma}\beta(\gamma\rho)\transpose(\nu\mathbf U)\varphi\mathrm d\omega\mathrm dt.
		\end{split}
		\end{equation}
		
		To use the Radon-Riesz-property again, we would like to choose $\sgn (\nu u)$ as a test function in order to receive a norm of $\gamma\rho$. But since this function is not smooth enough, we have to approximate it to ensure the $C^{0,1}$-regularity. Therefore, let be $\varphi_j\in C^{0,1}([0,T]\times\overline{\Omega},\mathbb R^n)$ such that
		\begin{equation*}
		\varphi_j\to\sgn(\nu u)
		\end{equation*}
		holds almost everywhere in $\Gamma_T$ with $|\varphi(t,x)|\leq 1$. The dominated convergence theorem immediately yields
		\begin{equation*}
		(\nu\mathbf U)\varphi_j\to|\nu u|
		\end{equation*}
		in $L^1(\Gamma_T)^n$. Thus, we find for $\varepsilon >0$ an index $J\in \mathbb N$ with
		\begin{equation*}
		\|(\nu\mathbf U)\varphi_J-|\nu u|\|_{L^1(\Gamma_T)^n}<\frac{\varepsilon}{4\bar\rho_{\max}^2}.
		\end{equation*}
		With the convergence in \eqref{eq:conv_gamma_rho} we find for $\varphi_J$ an index $K_1$ such that
		\begin{equation*}
		\left|\int_0^T\int_{\Gamma}\left(\beta(\gamma\rho_k)\transpose(\nu\mathbf U_k)-\beta(\gamma\rho)\transpose(\nu\mathbf U)\right)\varphi_J\mathrm d\omega\mathrm dt\right|<\frac{\varepsilon}{4}
		\end{equation*}
		for all $k\geq K_1$. Furthermore, because of the convergence of $\nu u_k$ in $L^1(\Gamma_T)^n$ there is also $K_2$ with
		\begin{equation*}
		\|\nu(u_k-u)\|_{L^1(\Gamma_T)^n}<\frac{\varepsilon}{4\bar\rho_{\max}^2}
		\end{equation*}
		for $k\geq K_2$.
		Finally, for all $k\geq \max(K_1,K_2)$ it holds
		\begin{align*}
		&\left|\|\gamma\rho_k\|_{L^2(\Gamma_T,|\mathrm d\mu_u|)}^2-\|\gamma\rho\|_{L^2(\Gamma_T,|\mathrm d\mu_u|)}^2\right|\\
		&\leq\left|\int_0^T\int_{\Gamma}\beta(\gamma\rho_k)\transpose(\nu \mathbf U)\left(\sgn(\nu u)-\varphi_J\right)\mathrm d\omega\mathrm dt\right|\\
		&\quad+\left|\int_0^T\int_{\Gamma}\beta(\gamma\rho_k)\transpose(\nu\left( \mathbf U-\mathbf U_k\right))\varphi_J\mathrm d\omega\mathrm dt\right|\\
		&\quad+\left|\int_0^T\int_{\Gamma}\left(\beta(\gamma\rho_k)\transpose(\nu \mathbf U_k)-\beta(\gamma\rho)\transpose(\nu\mathbf U)\right)\varphi_J\mathrm d\omega\mathrm dt\right|\\
		&\quad+\left|\int_0^T\int_{\Gamma}\beta(\gamma\rho)\transpose(\nu \mathbf U)\left(\varphi_J-\sgn(\nu u)\right)\mathrm d\omega\mathrm dt\right|\\
		&< 2\bar \rho_{\max}^2\||\nu u|-(\nu\mathbf U)\varphi_J\|_{L^1(\Gamma_T)^n}+\bar\rho_{\max}^2\|\nu(u-u_k)\|_{L^1(\Gamma_T)^n}+\frac{\varepsilon}{4}\\
		&<\varepsilon
		\end{align*}
		and we can use the Radon-Riesz-property and the boundedness of $\gamma\rho_k$ to conclude the strong convergence in $L^p(\Gamma_T,|\mathrm d\mu_u|)$ for all $p\in[1,\infty)$, which finishes the proof.
	\end{proof}
	\begin{remark} 
		Boyer and Fabrie obtained in \cite{Boyer.2012} an equivalent result for the scalar transport equation with non-coupled boundary conditions. However, their proof technique is completely different and uses a mollifying procedure. 
	\end{remark}
	\section{Existence}
	\label{sec:existence}
	 Now, we are finally able to prove the existence of solutions of the transport equation \eqref{eq:transport_equation}. This will be done in two steps. In a first step, we will consider the equation with non-coupled boundary conditions. In this case, the existence of a solution is classical and it can be proven in several ways. One possibility is to use the method of characteristics for smooth data. Then, the first part of the previous theorem yields the existence for general data. In \cite{Boyer.2012}, a proof of the existence using a parabolic approximation is presented. In a second step, we will construct a solution for general boundary conditions by an iterative method, where we will use the first part of Theorem~\ref{theo:sequent_cont}.
	\begin{lemma}[Existence]\label{lemma:exist_G0}
		For $\mathcal G=0$ there exists a solution of the initial-boundary-value problem \eqref{eq:transport_equation} under the assumptions from Section~\ref{sec:transport_assump}.
	\end{lemma}
	\begin{proof}See e.g. \cite{Boyer.2012} or \cite{RoggensackDiss.2014}.
	\end{proof}
	Now that we know the existence of solutions in the case of non-coupled boundary conditions, we will provide a small lemma concerning the monotonicity of the solution operator in dependence on the inflow values. This lemma will turn out to be useful for the construction of the solution in the general case with coupled boundary conditions.
	\begin{lemma}\label{lemma:trans_monoton}
		For $\mathcal G=0$ let the assumptions from Section~\ref{sec:transport_assump} be true  and let the functions $\rho_{\inn,1},\rho_{\inn,2}\in L^{\infty}(\Gamma_T,\mathrm d\mu_u^-)$ be two different inflow values with
		\begin{equation*}
		\rho_{\inn,2}\geq \rho_{\inn,1}\geq 0
		\end{equation*}
		$\mathrm d\mu_u^-$-almost everywhere.
		Then, for the solutions $\rho_i$ of the transport equation \eqref{eq:transport_equation} with inflow values $(\nu\mathbf U)^-\rho_i=(\nu\mathbf U)^-\rho_{\inn,i}$ it holds
		\begin{equation*}
		\rho_2(t)\geq \rho_1(t)
		\end{equation*}
		almost everywhere in $\Omega$. Moreover, for their traces $\gamma\rho_i$ hold
		\begin{equation*}
		\gamma\rho_2\geq\gamma\rho_1
		\end{equation*}
		$|\mathrm d\mu_u|$-almost everywhere in $\Gamma_T$.
	\end{lemma}
	\begin{proof}
		Let $\rho_1$ and $\rho_2$ be the solutions corresponding to the inflow values $\rho_{\inn,1}$ and $\rho_{\inn,2}$. Recall that these solutions $\rho_1$ and $\rho_2$ exist due to Lemma~\ref{lemma:exist_G0}. Then, the difference $\rho=\rho_2-\rho_1$ with trace $\gamma\rho=\gamma\rho_2-\gamma\rho_1$ solves the homogeneous transport equation
		\begin{align*}
		\rho_t+(\mathbf U\rho)_x+\mathbf C\rho&=0&&\text{in $(0,T)\times\Omega$}\\
		\rho(0,\cdot)&=0&&\text{in $\Omega$}\\
		(\nu\mathbf U)^-\rho&=(\nu\mathbf U)^-\left(\rho_{\inn,2}-\rho_{\inn,1}\right)&&\text{on $\Gamma_T$}.
		\end{align*}
		Since the inflow value $\rho_{\inn,2}-\rho_{\inn,1}$ is non-negative $\mathrm d \mu_u^-$ almost everywhere, we can apply Lemma~\ref{lemma:upper_bound} to obtain the lower bound $\rho(t)\geq 0$ almost everywhere in $\Omega$ and $\gamma\rho(t,\omega)\geq 0$ for $|\mathrm d\mu_u|$-almost all $(t,\omega)\in \Gamma_T$.
	\end{proof}
	Equipped with this lemma, we can prove the existence of solutions for general boundary operators $\mathcal H$ using an iterative construction in the next theorem.
		\begin{theorem}[Existence]
		\label{theo:transp_exist}
			Under the assumptions \eqref{assump:ucf} on $u$, $c$ and $f$ and under the assumptions from Section~\ref{sec:transport_assump} on $\mathcal H$, there exists a solution of the initial-boundary-value problem \eqref{eq:transport_equation}.
		\end{theorem}
		\begin{proof}
			We will iteratively construct a sequence of functions converging weakly-$\star$ to a solution $\rho$ of the transport equation. The idea is to use the outflow value of a solution to compute the inflow value for the problem of the next iteration step. Due to the positivity of the operator $\mathcal G$, the obtained sequence will be monotonically increasing.
			
			Therefore, we define the starting value of the sequence $\gamma\rho_0=0$. From Lemma~\ref{lemma:exist_G0} we know that the problem
			\begin{align*}
				(\rho_k)_t+(\mathbf U\rho_k)_x+\mathbf C\rho_k&=f&&\text{in }(0,T)\times\Omega\\
				\rho_k(0)&=\rho_0&&\text{in }\Omega\\
				(\nu \mathbf U)^-\rho_k&=(\nu \mathbf U)^-\mathcal H(\gamma\rho_{k-1})&&\text{on }\Gamma_T
			\end{align*} 
			has a unique solution $\rho_k$ with trace $\gamma\rho_k$.  Once more, we want to use the stability theorem (Theorem \ref{theo:sequent_cont}), but this time with $\rho_{\inn,k}=\mathcal H(\gamma\rho_{k-1})$ and $\mathcal G_k=0$.
			
			As a first step, we have to check the requirements for this theorem. Therefore, we want to analyse the boundedness of the boundary value. From assumption \eqref{eq:exist_M} we know that there is a constant $\rho_{\max}\in \mathbb R^n_{>0}$ with
			\begin{align*}
			\rho_{\inn}\exp\left(-\int_0^t\alpha(s)\mathrm ds\right)-F+\mathcal G(F)+\mathcal G(\rho_{\max})&\leq \rho_{\max}\\
			\shortintertext{and}
			\rho_0&\leq \rho_{\max}
			\end{align*}
			almost everywhere. We want to show by induction that an analogue bound for the boundary value $\rho_{\inn,k}=\mathcal H(\gamma\rho_{k-1})$ and $\mathcal G_k=0$ is valid for all $k$, i.e.\ 
			\begin{equation*}
			\gamma\rho_k\exp\left(-\int_0^t\alpha(s)\mathrm ds\right)-F\leq \rho_{\max}.
			\end{equation*} The base case is clear since $\mathcal G$ and $\rho_{\inn}$ are positive and it holds $\gamma\rho_0=0$. Using the positivity of $\mathcal G$, the induction hypothesis and the assumption \eqref{eq:exist_M} on $\rho_{\max}$ we calculate for the inductive step
			\begin{align*}
				&\mathcal H(\gamma\rho_{k-1})\exp\left(-\int_0^t\alpha(s)\mathrm ds\right)\\
				&=\rho_{\inn}\exp\left(-\int_0^t\alpha(s)\mathrm ds\right)+\mathcal G(\rho_{\max})+\mathcal G(F)\\
				&\quad-\mathcal G\left(\rho_{\max}+F-\gamma\rho_{k-1}\exp\left(-\int_0^t\alpha(s)\mathrm ds\right)\right)\\
				&\leq \rho_{\inn}\exp\left(-\int_0^t\alpha(s)\mathrm ds\right)+\mathcal G(\rho_{\max})+\mathcal G(F)\\
				&\leq \rho_{\max}+F.
			\end{align*}
			Thus, Lemma~\ref{lemma:upper_bound} yields the desired bound for $\gamma\rho_k$ for all $k$. 
			
			The next aim is to show the convergence of the inflow values $\rho_{\inn,k}$. As we have just seen, this sequence is bounded and it is monotonically increasing, i.e.\ it holds \begin{equation*}
			\rho_{\inn,k+1}(t,\omega)\geq \rho_{\inn,k}(t,\omega)
			\end{equation*} for $\mathrm d\mu_u^-$-almost all $(t,\omega)\in \Gamma_T$. We will also prove this by induction. Due to the non-negativity of $f$ and Lemma~\ref{lemma:upper_bound} it is $\gamma\rho_1(t,\omega)\geq 0$ for $|\mathrm d\mu_u|$-almost all $(t,\omega)$ and thus, because of the positivity of $\mathcal G$ and $\rho_{\inn}$
			\begin{equation*}
			\rho_{\inn,2}=\mathcal H(\gamma\rho_1)\geq \rho_{\inn}=\mathcal H(\gamma\rho_0)=\rho_{\inn,1} 
			\end{equation*} 
			holds $\mathrm d\mu_u^-$-almost everywhere. For the inductive step, we will use the previous lemma about the monotonicity. Assume, it is $\rho_{\inn,k}\geq\rho_{\inn,k-1}$. Then, Lemma~\ref{lemma:trans_monoton} immediately yields $\gamma\rho_k\geq \gamma\rho_{k-1}$. Thus, the positivity of $\mathcal G$ leads to
			\begin{equation*}
				\rho_{\inn,k+1}-\rho_{\inn,k}=\mathcal H(\gamma\rho_k)-\mathcal H(\gamma\rho_{k-1})=\mathcal G(\gamma\rho_{k}-\gamma\rho_{k-1})\geq 0.
			\end{equation*}
			
			Since $(\rho_{\inn,k})_k$ is bounded and monotone, we can apply the monotone convergence theorem to conclude that there exists a function $q\in L^{\infty}(\Gamma_T,\mathrm d\mu_u^-)$ with
			\begin{equation*}
				\rho_{\inn,k}\to q
			\end{equation*}
			in $L^{p}(\Gamma_T,\mathrm d\mu_u^-)$. Since the sequence is bounded in $L^{\infty}(\Gamma_T,\mathrm d\mu_u^-)$, it is also weakly\nobreakdash-$\star$ convergent in $L^{\infty}(\Gamma_T,\mathrm d\mu_u^-)$. Hence, Theorem~\ref{theo:sequent_cont} is applicable and it yields the existence of a solution $\rho$ with trace $\gamma\rho$ of
			\begin{align*}
			\rho_t+(\mathbf U\rho)_x+\mathbf C\rho&=f&&\text{in }(0,T)\times\Omega\\
			\rho(0)&=\rho_0&&\text{in }\Omega\\
			(\nu\mathbf U)^-\rho&=(\nu\mathbf U)^-q&&\text{on }\Gamma_T.
			\end{align*}
			Furthermore, it holds
			\begin{equation*}
			\gamma\rho_{k}\xrightharpoonup{\star}\gamma\rho
			\end{equation*}
			in $L^{\infty}(\Gamma_T,|\mathrm d\mu_u|)$. The weak-$\star$ continuity of the operator $\mathcal G$ implies
			\begin{align*}
			\mathcal H(\gamma\rho_{k})\xrightharpoonup{\star}\mathcal H(\gamma\rho)
			\end{align*}
			in $L^{\infty}(\Gamma_T,\mathrm d\mu_u^-)$. Finally, we conclude
			\begin{align*}
			(\nu\mathbf U)^-\gamma\rho&\leftharpoonup(\nu\mathbf U)^-\gamma\rho_{k}\\
			&=(\nu\mathbf U)^-\mathcal H(\gamma\rho_{k-1})\\
			&\rightharpoonup(\nu\mathbf U)^-\mathcal H(\gamma\rho)
			\end{align*}
			in $L^1(\Gamma_T)$. Because of the uniqueness of the weak limit the boundary condition
			\begin{equation*}
			(\nu\mathbf U)^-\gamma\rho=(\nu\mathbf U)^-\mathcal H(\gamma\rho)
			\end{equation*}
			is satisfied.
		\end{proof}
		
\begin{remark}
	Remark, that the required non-negativity of $f$ is essential in the previous proof. In order to prove the existence without this restriction, one needs an additional requirement on $\mathcal G$, e.g.\ the existence of a constant vector $\hat \rho_{\min}\in\mathbb R^n_{\geq0}$ with $\mathcal G(F_-+\hat{\rho}_{\min})\leq F_-+\hat \rho_{\min}$ for $\mathrm d\mu_u^-$ almost all $(t,\omega)\in\Gamma_T$. Here, it is $F_-(t)=\exp\left(-\int_0^t\alpha(s)\mathrm ds\right)\begin{pmatrix}
\|(f_1(s,\cdot))^{-}\|_{L^{\infty}}\\
\vdots\\
\|(f_n(s,\cdot))^{-}\|_{L^{\infty}}
\end{pmatrix}\mathrm ds.$
Then, we can split the problem in the two subproblems
\begin{align*}
\rho^1_t+(\mathbf U\rho^1)_x+\mathbf C\rho^1&=f^+&&\text{in $(0,T)\times\Omega$}\\
\rho^1(0,\cdot)&=\rho_0&&\text{in $\Omega$}\\
(\nu\mathbf U)^-\rho^1&=(\nu\mathbf U)^-\mathcal H(\rho^1|_{\Gamma_T})&&\text{on $\Gamma_T$}\\
\shortintertext{and}
\rho^2_t+(\mathbf U\rho^2)_x+\mathbf C\rho^2&=f^-&&\text{in $(0,T)\times\Omega$}\\
\rho^2(0,\cdot)&=0&&\text{in $\Omega$}\\
(\nu\mathbf U)^-\rho^2&=(\nu\mathbf U)^-\mathcal G(\rho^2|_{\Gamma_T})&&\text{on $\Gamma_T$}
\end{align*}
so that $\rho=\rho^1-\rho^2$ is a solution of the original problem. Due to the existence of $\hat \rho_{\min}$, both subproblems fulfill all requirements of section \ref{sec:transport_assump}.
\end{remark}
	\section{Application on networks}
	\label{sec:application}
	One of the main applications of such coupled boundary value problems are transport problems on networks or graphs. In this case, the boundary operator describes the coupling of the density at the nodes of the network. As usual, we enforce the mass to be conserved at the nodes. Furthermore, the density should be distributed by a prescribed ratio. This is similar to the coupling conditions for traffic flow (see \cite{garavello2006}) or open channel flow (see \cite{cunge1980} or \cite{roggensack2013}).
	
	In order to describe such flows on a network, the notation of graph theory is very well suited. We will briefly introduce the notions which are needed in this context. For more details, we refer to the PhD thesis \cite{RoggensackDiss.2014} and to any text book about graph theory as e.g.\ \cite{Diestel.2006}.
	
Let $G=(V,E,w,\init,\ter)$ be an oriented, weighted and connected graph with $n$ edges and $m$ nodes and let $\mathbf B\in\{-1,0,1\}^{m\times n}$ denote its incidence matrix. Assume that the graph has $k>0$ inner nodes $v_j\in V$ with degree $d(v_j)>1$. We denote by $\mathbf B_{>1}\in\{-1,0,1\}^{k\times n}$ the submatrix of the incidence matrix containing only the rows of $\mathbf B$ corresponding to inner nodes. In the same way, $\mathbf B_{=1}\in\{-1,0,1\}^{(m-k)\times n}$ is the submatrix corresponding to the outer nodes $v_j\in V$ with degree $d(v_j)=1$.

With this notation, we are able to write the mass conservation at the inner nodes as
\begin{equation*}
\int_{\Gamma}\left(\nu\mathbf B_{>1}\right)^-\left(\nu\mathbf U\right)\rho\mathrm d\omega=0.
\end{equation*}
If there is more than one outflow edge at the same node, we need to specify the distribution of the density. The simplest model is the perfect mixing of the density, i.e.\ the density in the outflow edges is continuous at the node. However, also other linear models to determine the distribution on the different outflow edges are possible.

In order to define the above-described boundary operator for the perfect mixing precisely, we introduce the matrix $\mathbf M\in L^1((0,T))^{k\times k}$ by
\begin{equation}
\mathbf M=\int_{\Gamma}\left(\nu\mathbf B_{>1}\right)^-\left(\nu\mathbf U\right)^-\left(\nu\mathbf B_{>1}\transpose\right)^-\mathrm d\omega.
\end{equation}
A simple computation using the structure of the incidence matrix shows that this matrix $\mathbf M=(m_{ij})$ has a diagonal form with
	\begin{equation*}
	m_{ij}=\delta_{ij}\int_{\Gamma}\left(\left(\nu\mathbf B_{>1}\right)^-(\nu u)^-\right)_i\mathrm d\omega
	\end{equation*}
	(see \cite{RoggensackDiss.2014}). For the well-posedness of the transport equation we need additionally some sort of energy conservation. We assume
	\begin{equation}
	\label{eq:weak_energy:cons}
\int_{\Gamma}\left(\left(\nu\mathbf B_{>1}\right)^-(\nu u)\right)_i\mathrm d\omega = 0
	\end{equation}
	for all inner nodes $v_i\in V$ with $d(v_i)>1$. In the formal low Mach number limit of the Euler or Navier-Stokes equation on a network studied in \cite{RoggensackDiss.2014} this condition is fulfilled naturally.
	
Then, for $\rho_{\out}\in L^{\infty}((0,T))^{m-k}$ with $\rho_{\out}\geq \bar \rho_{\min}>0$ almost everywhere, we define the boundary operator
\begin{align*}
\mathcal H_u\colon L^{\infty}(\Gamma_T,\mathrm d\mu_u^+)&\to L^{\infty}(\Gamma_T,\mathrm d\mu_u^-)\\
\rho&\mapsto \rho_{\inn}+\mathcal G_u(\rho)
\end{align*}
with
\begin{align*}
\rho_{\inn}&=\left(\nu\mathbf B_{=1}\transpose\right)^-\rho_{\out}
\shortintertext{and}
\mathcal G_u&\colon L^{\infty}(\Gamma_T,\mathrm d\mu_u^+)\to L^{\infty}(\Gamma_T,\mathrm d\mu_u^-)\\
\mathcal G_u(\rho)&=\left(\nu \mathbf B_{>1}\transpose\right)^-\mathbf M^{-1}\int_{\Gamma}\left(\nu\mathbf B_{>1}\right)^-\left(\nu\mathbf U\right)^+\rho\mathrm d\omega.
\end{align*}
Here and in the following, $\mathbf M^{-1}$ has to be understood in the sense of the Moore-Penrose pseudoinverse.
 Clearly, this boundary operator satisfies the assumption on the perfect mixing since for each inner node there is a common outflow value for the density. Furthermore, the mass is conserved at the inner node since it holds
 \begin{equation}
 \label{eq:H_mass_cons}
\begin{split}
 \int_{\Gamma}\left(\nu\mathbf B_{>1}\right)^-\left(\nu\mathbf U\right)^-\mathcal H_u(\rho)\mathrm d\omega&= \int_{\Gamma}\left(\nu\mathbf B_{>1}\right)^-\left(\nu\mathbf U\right)^-\mathcal G_u(\rho)\mathrm d\omega\\
 &=\mathbf M\mathbf M^{-1}\int_{\Gamma}\left(\nu\mathbf B_{>1}\right)^-\left(\nu\mathbf U\right)^+\rho\mathrm d\omega\\
 &=\int_{\Gamma}\left(\nu\mathbf B_{>1}\right)^-\left(\nu\mathbf U\right)^+\rho\mathrm d\omega
\end{split}
 \end{equation}
 for all $\rho\in L^{\infty}(\Gamma_T,\mathrm d\mu_u^+)$ and almost all $t\in[0,T]$. For the first equality we used the fact $(\nu_1\mathbf B_{>1})^-\mathbf W(\nu_2\mathbf B_{=1}\transpose)^-=0$ for arbitrary diagonal matrices $\mathbf W$ and $\nu_1,\nu_2\in\{-1,1\}$. And for the last equality we have to take into account the energy conservation \eqref{eq:weak_energy:cons} since $\mathbf M^{-1}$ is just the pseudoinverse: In the case $m_{ii}=0$ it is also
\begin{equation*}
 \int_{\Gamma}\left(\left(\nu\mathbf B_{>1}\right)^-(\nu u)^+\right)_i\mathrm d\omega = 0
\end{equation*}
by \eqref{eq:weak_energy:cons}. 
This is a sum of positive terms and thus all summands have to be zero and the third equality in \eqref{eq:H_mass_cons} is obvious.
 
Until now we have formally seen that the operator describes the above-mentioned coupling conditions. Before we check that $\mathcal H_u$ fulfils the requirements of Section~\ref{sec:transport_assump} for the well-posedness of the problem, we will show the well-definedness of $\mathcal G_u$ and we will introduce the pre-adjoint operator $\bar{\mathcal G}_u$, to which $\mathcal G_u$ is the adjoint operator.

	For all $\rho\in L^{\infty}(\Gamma_T,\mathrm d\mu_u^+)$ and for almost all $t\in [0,T]$ it is either $m_{jj}(t)=0$ or
	\begin{equation}
	\label{eq:ineq_m}
	m_{jj}(t)^{-1}\int_{\Gamma}\left(\left(\nu\mathbf B_{>1}\right)^-\left(\nu\mathbf U(t)\right)^+\rho(t,\omega)\right)_j\mathrm d\omega\leq \max_{j,\omega}(\rho_j(t,\omega)).
	\end{equation}
	In the case $m_{jj}(t)=0$, it is $\left(\mathcal G_u\right)_i(\rho)(t,\omega)=0$ for all adjacent edges $e_i$.
	
	Altogether, this proves
	\begin{equation*}
	\left|\left(\mathcal G_u\right)_i(\rho)\right|\leq \max_{j,\omega}\left(|\rho_j(t,\omega)|\right)
	\end{equation*}
	for almost all $(t,\omega)\in \Gamma_T$ with $\nu u_i(t,\omega)<0$.
	Hence, it is $\mathcal G_u(\rho)\in L^{\infty}(\Gamma_T,\mathrm d\mu_u^-)$ and the linear operator $\mathcal G_u$ is well-defined. Descriptively, this is not surprising since the coupling condition chooses a weighted mean value of the inflow densities for each $t$. Clearly, this weighted mean value is smaller than the maximum.
	
As a next step, we define the operator $\bar{\mathcal G}_u$ as
\begin{align*}
\bar{\mathcal G}_u\colon L^1(\Gamma_T,\mathrm d\mu_u^-)&\to L^1(\Gamma_T,\mathrm d\mu_u^+)\\
\bar{\rho}&\mapsto\left(\nu \mathbf B_{>1}\transpose\right)^-\mathbf M^{-1}\int_{\Gamma}\left(\nu\mathbf B_{>1}\right)^-\left(\nu\mathbf U\right)^-\bar \rho\mathrm d\omega.
\end{align*}
This operator is well-defined by the same argument.
 We observe that $\mathcal G_u$ is the adjoint operator of $\bar{\mathcal G}_u$. To check this, let be $\rho\in L^{\infty}(\Gamma_T,\mathrm d\mu_u^+)$ and $\bar \rho\in L^1(\Gamma_T,\mathrm d\mu_u^-)$. Then, we compute
	\begin{align*}
	\int_{\Gamma_T}\bar{\rho}\transpose(\nu\mathbf U)^-\mathcal G_u(\rho)\mathrm d\omega\mathrm dt&=\int_0^T\int_{\Gamma}\bar{\rho}\transpose(\nu\mathbf U)^-\left(\nu \mathbf B_{>1}\transpose\right)^-\mathrm d\omega\mathbf M^{-1}\int_{\Gamma}\left(\nu\mathbf B_{>1}\right)^-\left(\nu\mathbf U\right)^+\rho\mathrm d\omega\\
	&=\int_{\Gamma_T}\bar{\mathcal G}_u(\bar \rho)\transpose(\nu\mathbf U)^+\rho\mathrm d\omega\mathrm dt.
	\end{align*}
 Now, we will check all requirements for the boundary operator assumed in Section~\ref{sec:transport_assump}.
\begin{lemma}[Continuity equation on a network]
Let $u\in L^1((0,T),W^{1,1}(\Omega)^{n})$ with $u_x\in L^1((0,T),L^{\infty}(\Omega)^n)$ and $\rho_0\in L^\infty(\Omega)^n$ be given and let be $\mathcal H_u$ the above-defined operator. Then, the continuity equation on a network
\begin{align*}
\rho_t+\left(\mathbf U\rho\right)_x&=0&\text{in }(0,T)\times\Omega\\
	\label{eq:transport_equation}
	\rho(0,\cdot)&=\rho_0&\text{in }\Omega\\
	\notag
	(\nu \mathbf U)^-\rho&=(\nu \mathbf U)^-\mathcal H_u(\rho|_{\Gamma_T})\quad&\text{on }\Gamma_T
\end{align*}
has a unique solution which is bounded and depends continuously on $\rho_0$ and $\rho_{\inn}$. In addition, for each $\rho_0\in L^{\infty}(\Omega)^n$ with $\rho_0\geq \bar\rho_{\min}>0$ almost everywhere, there is a vector $\rho_{\min}\in\mathbb R^n_{>0}$ satisfying inequality \eqref{eq:exist_m}, which guarantees the existence of a strictly positive lower bound of the solution.
\end{lemma}
\begin{proof}
We will verify each of the assumptions of Section~\ref{sec:transport_assump} so that we can apply the theory of the previous sections. Therefore, let be $\rho\in L^{\infty}(\Gamma_T,\mathrm d\mu_u^+).$

\textbf{Weak-$\star$ continuity:} Since $\mathcal G_u$ is the adjoint of another operator, namely $\bar{\mathcal G}_u$, it is weak-$\star$ continuous.

\textbf{$L^1$-operator norm:} As weight matrix for the $L^p(\Gamma_T,\mathrm d\mu_u^{\pm})$ norm we choose 
	the identity matrix $\overline{\mathbf W}=\mathbf I$. Then, using the definition of $\mathbf M$, we estimate
	\begin{align*}
	\|\mathcal G_u(\rho)\|_{1,w,-}&=\int_{\Gamma_T}|\mathcal G_u(\rho)|\transpose(\nu u)^-\mathrm d\omega\mathrm dt\\
	&=\int_0^T\left|\int_{\Gamma}\rho\transpose(\nu \mathbf U)^+\left(\nu\mathbf B_{>1}\transpose\right)^-\mathrm d\omega\right|\mathbf M^{-1}\int_{\Gamma}\left(\nu\mathbf B_{>1}\right)^-(\nu u)^-\mathrm d\omega\mathrm dt\\
	&=\int_0^T\left|\int_{\Gamma}\rho\transpose(\nu \mathbf U)^+\left(\nu\mathbf B_{>1}\transpose\right)^-\mathrm d\omega\right|\Eins\mathrm dt\\
	&\leq \|\rho\|_{1,w,+}.
	\end{align*}
	Thus, the $L^1$-operator norm of $\mathcal G_u$ is less or equal one.
	
	\textbf{Causality:}
	Equation \eqref{eq:G_chi} is fulfilled for almost all $t\in[0,T]$ since it holds
	\begin{equation*}
	\chi_{[0,t]}\mathcal G_u(\chi_{[0,t]}\rho)=\left(\nu \mathbf B_{>1}\transpose\right)^-\mathbf M^{-1}\int_{\Gamma}\left(\nu\mathbf B_{>1}\right)^-\left(\nu\mathbf U\right)^+\chi_{[0,t]}\rho\mathrm d\omega=\chi_{[0,t]}\mathcal G_u(\rho)
	\end{equation*}
	by definition of $\mathcal G_u$.
	
	\textbf{Positivity:} Let be $\rho\geq0$ almost everywhere. Then, $\mathcal G_u(\rho)$ is positive as sum of products of positive factors.
	
	\textbf{Boundedness:} For any $\rho_0\in L^{\infty}(\Omega)^n$ define the scalar \begin{equation*}
	\rho_{\max}=\max\left(\|\rho_{\out}r\|_{L^{\infty}(\Gamma_T)^{m-k}},\|\rho_0\|_{L^{\infty}(\Omega)^n}\right)
	\end{equation*}
	with $r=\exp\left(-\int_0^t\|(u_x(s,\cdot))^-\|_{L^{\infty}(\Omega)^n}\mathrm ds\right)$. Observe that the reaction term $c$ and the source term $f$ are not present in the continuity equation. Due to the definition of the matrix $\mathbf M$, it is
	\begin{equation*}
	\mathcal G_u(\rho_{\max}\Eins_n)=\rho_{\max}\mathcal G_u(\Eins_n)=\rho_{\max}\left(\nu\mathbf B_{>1}\transpose\right)^-\mathbf M^{-1}\mathbf M\Eins_k=\rho_{\max}\left(\nu\mathbf B_{>1}\transpose\right)^-\Eins_k
	\end{equation*}
	and thus, it follows also
	\begin{align*}
		\rho_{\inn}r+\mathcal G_u(\rho_{\max}\Eins_n)&=\left(\nu\mathbf B_{=1}\transpose\right)^-\rho_{\out}r+\mathcal G_u(\rho_{\max}\Eins_n)\\
		&\leq \rho_{\max}\left(\nu\mathbf B_{=1}\transpose\right)^-\Eins_{m-k}+\rho_{\max}\left(\nu\mathbf B_{>1}\transpose\right)^-\Eins_k\\
		&=\rho_{\max}(\nu\mathbf B\transpose)^-\Eins_m\\
		&=\rho_{\max}\Eins_{n}
	\end{align*}
	$\mathrm d\mu_u^-$-almost everywhere. Here, one has to be careful with the different dimensions of the $\Eins$-vectors. In the last step, we used that $(\nu\mathbf B)^-$ has in each column exactly one non-vanishing entry.

	Now, we additionally assume $\rho_0\geq \rho_{\min}>0$ almost everywhere and denote \begin{equation*}
	\bar r(t)=\exp\left(\int_0^t\|(u_x(s,\cdot))^+\|\mathrm ds\right).
	\end{equation*} Then, we define
	\begin{equation*}
	\rho_{\min}=\min\{\ess\inf_{\Omega} \rho_0,\ess\inf_{\Gamma_T} (\bar r\rho_{\out})\} \geq \bar\rho_{\min}>0
	\end{equation*}
	and thus, it holds
	\begin{align*}
	\rho_{\inn}\bar r+\mathcal G_u(\rho_{\min}\Eins_n)&=\left(\nu\mathbf B_{=1}\transpose\right)^-\rho_{\out}\bar r+\rho_{\min}\mathcal G_u(\Eins_n)\\
	&\geq  \rho_{\min}\left(\nu\mathbf B_{=1}\transpose\right)^-\Eins_{m-k}+\rho_{\min}\left(\nu\mathbf B_{>1}\transpose\right)^-\Eins_k\\
	&=\rho_{\min}(\nu\mathbf B\transpose)^-\Eins_m\\
	&=\rho_{\min}\Eins_{n}
	\end{align*}
	$\mathrm d\mu_u^-$-almost everywhere. This completes the proof.

\end{proof}
\begin{remark}
Using the stability theorem \ref{theo:sequent_cont}, one can also proof that the solution of the equations on a network continuously depends on the velocity $u$.
\end{remark}

	\bibliographystyle{plain}
	\bibliography{../../bibliothek}

\end{document}